\documentclass{amsart}
\usepackage{bm}
\usepackage{amssymb}
\usepackage{scrextend}
\usepackage{amsthm}
\usepackage{microtype}
\usepackage{tikz}
\usepackage{framed}
\usepackage{theoremref}
\usepackage{comment}
\usepackage{subfigure}

\usepackage{rotating}
\usepackage{changepage}
\usepackage[showframe=false, margin = 1.5 in]{geometry}
\usepackage{paralist}
\usepackage{enumitem}

\DeclareMathOperator{\ESF}{ESF}
\newcommand{\ESFn}{\ESF(\alpha,n)}
\renewcommand{\vec}{\mathbf }

\renewcommand{\emptyset}{\varnothing}
\newcommand{\E}{\mathbf{E}}
\renewcommand{\P}{\mathbf{P}}

\renewcommand{\emptyset}{\varnothing}

\newcommand{\inv}{^{-1}}
\newcommand{\f}{\frac}
\newcommand{\ind}[1]{\mathbf{1}{\{#1\}} }

\usepackage{color}
\definecolor{dark_green}{RGB}{1, 180, 1}

\DeclareMathOperator{\Poi}{Poi}

\usepackage{mathrsfs}
\renewcommand{\l}{\ell}
\newcommand{\C}{\mathcal C}

\theoremstyle{plain}
\newtheorem{thm}{Theorem}
\newtheorem{lemma}[thm]{Lemma}
\newtheorem{prop}[thm]{Proposition}
\newtheorem{cor}[thm]{Corollary}

\newtheorem{question}{Question}
\newtheorem*{claim}{Claim}

\theoremstyle{remark}
\newtheorem{remark}[thm]{Remark}

\newcommand{\Y}{\mathbf Y}
\renewcommand{\C}{\mathbf C}
\renewcommand{\L}{\mathscr L}
\newcommand{\X}{\mathbf X}

\newcommand\m\mathbb
\newcommand\ve\varepsilon
\newcommand{\salp}{strong $\alpha$-logarithmic property}

\newcommand{\HOX}[1]{\marginpar{\footnotesize #1}}

\renewcommand\epsilon\ve

\author[G. Brito]{Gerandy Brito}
\address{Department of Mathematics, University of Washington}
\email{gerandy@math.washington.edu}

\author[C. Fowler]{Christopher Fowler}
\address{Department of Mathematics, University of Washington}
\email{cff2008@math.washington.edu}

\author[M. Junge]{Matthew Junge}
\address{Department of Mathematics, Duke University}
\email{jungem@math.duke.edu}

\author[A. Levy]{Avi Levy}
\address{Department of Mathematics, University of Washington}
\email{avius@math.washington.edu}

\begin{document}

\title{Ewens sampling and invariable generation}

\begin{abstract}
We study the number of random permutations needed to invariably generate 
the symmetric group, $S_n$, when the distribution of cycle counts has the strong $\alpha$-logarithmic property. The canonical example is the Ewens sampling formula, for which the number of $k$-cycles relates to a conditioned Poisson random variable with mean $\alpha/k$. The special case 
 $\alpha =1$ corresponds to uniformly random permutations, for which it was recently shown that exactly four are needed.
 
 For strong $\alpha$-logarithmic measures, and almost every $\alpha$, 
  we show that precisely $\left\lceil ( 1- \alpha \log 2 )^{-1}  \right\rceil$ permutations are needed to invariably generate $S_n$.
   A corollary is that for many other probability measures on $S_n$ no bounded number of permutations will invariably generate $S_n$ with positive probability.
%
Along the way we generalize classic theorems of Erd\H{o}s, Tehran, Pyber, Luczak and Bovey to permutations obtained from the Ewens sampling formula.   






\end{abstract}

\maketitle

\tableofcontents

\section{Introduction}

In 1934, van der Waerden raised the question of the minimal number of uniformly random permutations required to invariably generate $S_n$ with positive probability \cite{Waerden1934}. 
Recent breakthroughs by Pemantle, Peres, Rivin \cite{ppr} and Eberhard, Ford, Green \cite{efg} resolved van der Waerden's question, showing that precisely four permutations are required.
We consider the same question in the context of more general permutation measures related to the cycle structure, primarily the Ewens sampling formula (1) introduced in [Ewe72]. %
It aligns a permutation's cycle structure with conditioned independent Poisson random variables.


Elements $g_1,\ldots,g_m$ of a group $G$ {\em generate} the group if the smallest subgroup containing them is $G$ itself. The {\em symmetric group} $S_n$ is the group of permutations of the set $[n]:=\{1,\ldots,n\}$. Permutations $\pi_1,\ldots,\pi_m$ {\em invariably generate} $S_n$ if for all $\sigma_1,\ldots,\sigma_m\in S_n$ the elements $\sigma_1\pi_1\sigma_1^{-1},\ldots,\sigma_m\pi_m\sigma_m^{-1}$ generate $S_n$.
%
%
%
The definition of invariable generation is conjugacy invariant. Conjugacy classes of the symmetric group are parametrized by the {\em cycle type} of a permutation, defined to be the multiset of cycle lengths in the cycle decomposition of a permutation. For this reason, invariable generation concerns only the cycle structure of permutations, rather than the action on $[n]$.
%

Recent developments show that cycle counts converge to independent Poisson random variables for many measures on $S_n$ as $n$ tends to infinity.
Arratia, Barbour and Tavare have shown that strong $\alpha$-logarithmic measures have this property and numerous further results \cite{arratia, arratia1992, arratia2000, arratia2016}. 
Measures with general cycle weights are considered in \cite{weights1, weights2} and related permutons are introduced in \cite{permutons1}. 
Following this work, Mukherjee used Stein's method to deduce that a variety of permutation measures have Poisson limiting cycle counts \cite{permutons}. 
Furthermore, the structure of the short cycles in the Mallows measure was characterized in \cite{peled}.


 
In addition to its intrinsic interest, the minimal number of permutations required for invariable generation has applications to computational Galois theory. 
One may use this quantity to estimate the runtime of a Monte Carlo polynomial factorization algorithm \cite{g1,g3,g2,ppr}. 
This establishes a connection between the irreducible factors of polynomials and the cycle structure of permutations. 
The cycle structure of a permutation is also compared with prime factorization of integers in the survey \cite{anatomy}. 
An analogy between divisors of large integers and invariable generation of permutations is then mentioned in the introduction of \cite{efg}, postulating a link between the threshold for invariable generation and small divisors of a set of random integers.



 Despite the age of van der Waerden's question, it was only recently shown that the number of uniformly random permutations required to invariably generate $S_n$ is bounded. 
 A bound of $O(\sqrt{\log n})$ was first established in \cite{dixon}. 
 This was improved to $O(1)$ in \cite{luczak}, but the constant was large ($\approx 2^{100}$). 
 Pemantle, Peres and Rivin reduced the bound to four in \cite{ppr}, which was then shown to be sharp by Eberhard, Ford and Green in \cite{efg}.


 We extend the results for uniform permutations to the Ewens' sampling formula and beyond.
Along the way, we extend results in \cite{erdos,luczak,bovey1980probability}, replacing counting arguments with a tool known as the Feller coupling \eqref{eq:feller} introduced in \cite[p.\ 815]{feller1945fundamental}.




\subsection{Ewens sampling and the logarithmic property}

The Ewens sampling formula appears throughout mathematics, statistics and the sciences. 
In the words of Harry Crane it, ``exemplifies the
harmony of mathematical theory, statistical application,
and scientific discovery." This is stated in his survey \cite{crane}. It and its followup \cite{crane2} give a nice tour of the formula's universal character, describing applications to evolutionary molecular genetics, the neutral theory of biodiversity, Bayesian nonparametrics, combinatorial stochastic processes, and inductive inference, to name a few. The sampling formula also underpins foundational mathematics in number theory, stochastic processes, and algebra.

We start by formally defining the \emph{Ewens sampling formula}, as well as the \salp. 
Let $\Poi(\lambda)$ denote the law of a Poisson random variable with mean $\lambda$. Fix $\alpha>0$ and let $X_1,\ldots,X_n$ be independent random variables in which $X_k$ has law $\Poi(\alpha/k)$. For a permutation $\pi\in S_n$ let $C_k$ denote the number of $k$-cycles in $\pi$. More precisely, $C_k$ is the number of disjoint $k$-element subsets of $[n]$ on which $\pi$ restricts to a cyclic permutation.
%
 We say that the vector $\vec C = (C_1,\hdots, C_n)$ has distribution $\ESFn$ if it satisfies the joint distribution:
\begin{align}
\P[ C_1 = x_1, \hdots, C_n = x_n ] &= \P \Big[ X_1 = x_1, \hdots, X_n = x_n \; \Big|\; \sum_1^n k X_k = n \Big] 	 \label{eq:ESF}.
\end{align}

  The tuple $(C_1,\ldots,C_n)$ uniquely specifies the cycle type of a permutation. 
  Furthermore, it is easy to see that as $\pi$ ranges over $S_n$, the tuple $(C_1,\ldots,C_n)$ associated to $\pi$ ranges over the set
  $$
    \left\{(C_1,\ldots,C_n)\in \{0,\ldots,n\}^n\colon  \sum_{k=1}^nkC_k=n\right\}.
  $$
  Thus, \eqref{eq:ESF} specifies a measure on the conjugacy classes of $S_n$. We extend this to a measure on $S_n$ (also denoted by ESF$(\alpha,n)$) by weighting all elements of the same class uniformly. Note that ESF$(1,n)$ is the uniform measure, and therefore the results we establish regarding ESF$(\alpha,n)$ will generalize the corresponding results about uniform permutations appearing in \cite{ppr,efg}.


A more general family that includes $\ESFn$ is obtained by replacing the $X_k$ in \eqref{eq:ESF} with an arbitrary sequence of independent nonnegative integer valued random variables $Z_1,Z_2,\hdots$. To avoid intractable distortions when conditioning on $\sum k Z_k = n$, logarithmic growth of $\sum Z_k$ is traditionally assumed (see \cite{arratia2000}). 
For our purposes, we require the $Z_k$ satisfy the \emph{strong} $\alpha$-\emph{logarithmic condition}:
\begin{align}
|i \P[Z_i = 1]-\alpha| &< e(i) c_1, \label{eq:ULC1}\\
i \P[Z_i = \ell	] &\leq e(i) c_\ell, \qquad l \geq 2, \nonumber \\
&\text{where $\textstyle \sum i^{-1}e(i)$ and $\textstyle\sum \ell c_\ell$ are both finite}. \nonumber
\end{align}
Besides Poisson, the most commonly used distributions for the $Z_i$ are the negative-binomial and binomial distributions. These are called assemblies, multisets, and selections, respectively (see \cite{arratia2000} for more discussion). The proof that all three of these different distributional choices for the $Z_i$ satisfy the strong $\alpha$-logarithmic property is in \cite[Proposition 1.1]{arratia2000}. 

It is no trouble to work at this level of generality, because the limiting cycle structure of any family satisfying the {\salp}~ is the same as that for an $\ESFn$ permutation. We will describe this in more detail in Section \ref{sec:overview}, but in short, the number of $\ell$-cycles converges to an independent Poisson with mean $\alpha/\ell$.

\subsection{Statement of \thref{thm:ig}}
We state our result in terms of random variables $Z_1, Z_2, \hdots$ satisfying the strong $\alpha$-logarithmic property at \eqref{eq:ULC1}. The example to keep in mind though is the measure $\ESFn$, since all of these have the same limiting cycle counts.

 Given such a collection let $\mu_n(\alpha,Z_1,Z_2,\hdots,Z_n)$ be the measure induced on $S_n$ via the relation at \eqref{eq:ESF} with $X_k$ replaced by $Z_k$. We define $m_\alpha = m_\alpha(Z_1,Z_2,\hdots)$ to be the minimum number of permutations sampled according to $\mu_n$ to invariably generate $S_n$ with positive probability as $n \to \infty$:
\begin{align}m_\alpha = \inf_{m \geq 2} \Big\{ m \colon \inf_{n \geq 2} \P\big [\{\pi_1,\hdots,\pi_m\} \subset S_n, \pi_i \sim \mu_n, \text{ invariably generate $S_n$}\big] > 0 \Big\}.
\label{eq:m}	
\end{align}
Our theorem gives a closed formula for $m_\alpha$ (see Figure \ref{fig:1}), save for a countable exceptional set. Because we will reference it several times, we write the expression here:
\begin{align}h(\alpha) = \begin{cases}\left \lceil (1- \alpha \log 2)^{-1} \right\rceil, &0 <\alpha< 1 /\log 2\\
	\infty, & \alpha \geq 1 /\log 2
	  \end{cases} . \label{eq:h}
\end{align}

\begin{thm} \thlabel{thm:ig}
	Let $Z_1,Z_2,\hdots$ be a collection of random variables satisfying the strong $\alpha$-logarithmic condition at \eqref{eq:ULC1}. Let $m_\alpha=m_\alpha(Z_1,Z_2,\hdots)$, as in \eqref{eq:m}, be the minimum number of permutations 
	to invariably generate $S_n$ with positive probability. Let $h(\alpha)$ be as in \eqref{eq:h}. For points of continuity of $h$ 
	  it holds that $m_\alpha = h(\alpha)$.
At points of discontinuity we have $h(\alpha) \leq m_\alpha \leq h(\alpha)+1$.

\end{thm}

\begin{figure}
	\includegraphics[width = 10 cm]{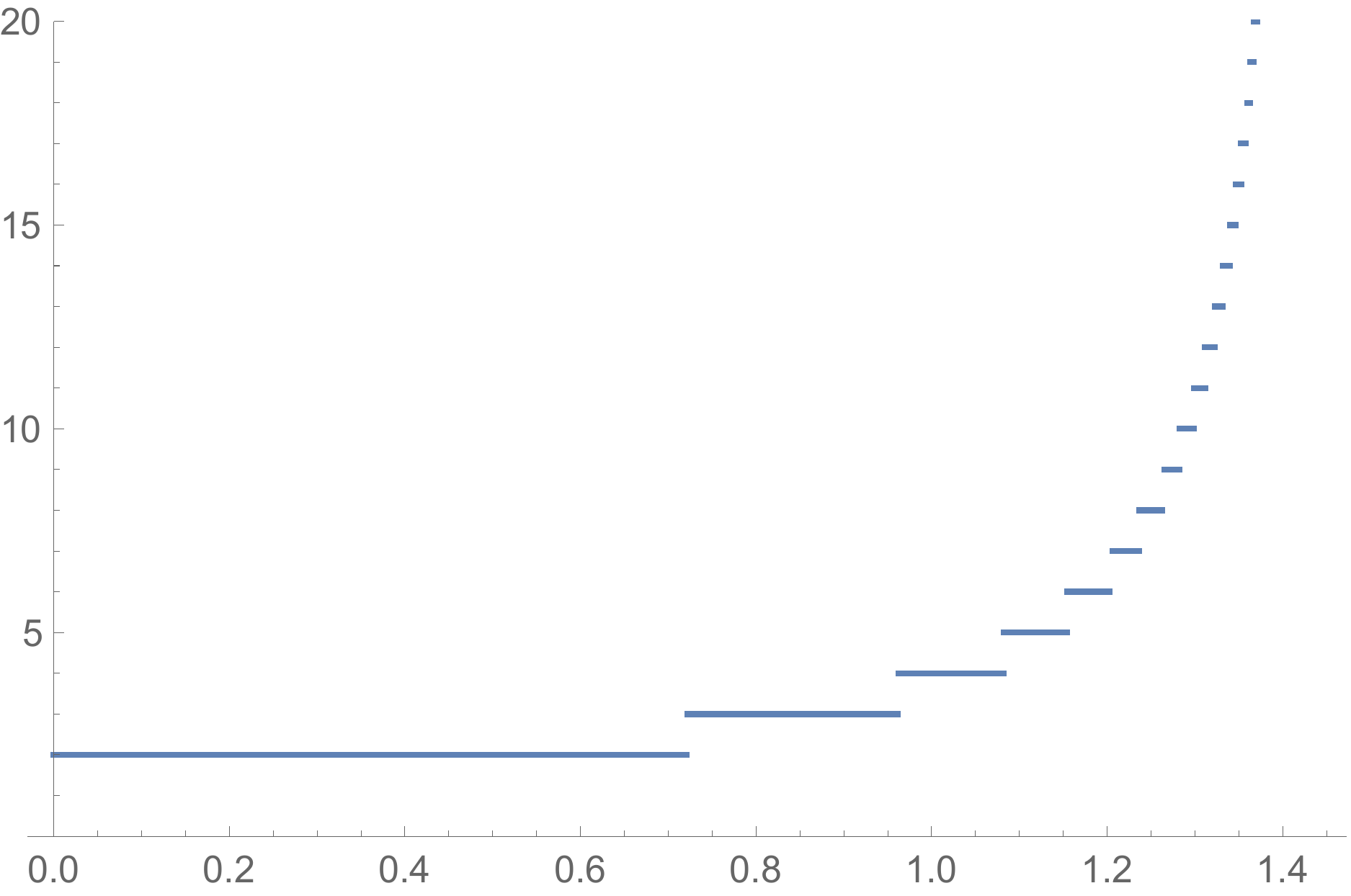}
	\caption{The graph $h(\alpha)$ defined at \eqref{eq:h}, and thus the behavior of $m_\alpha$ and $s_\alpha$. We are not sure about the value at the points of discontinuity in $h$. See \thref{q:critical} for more discussion.} \label{fig:1}
\end{figure}


This has broader implications beyond measures that have the \salp. We use it to deduce that any finite collection of random permutations with cycle counts asymptotically dominating a $\Poi(1/j\log 2)$ random variable will fail to generate a transitive subgroup of $S_n$. 

\begin{cor} \thlabel{cor:lb}
Suppose that $c_j \geq 1/\log 2$ and $X_k \sim \Poi(c_j/j)$. Let $\kappa_n$ be a sequence of probability measures on $S_n$ such that for random permutations $\pi \sim \kappa_n$ the cycle structure, $\C$ satisfies $$d\left( (C_1,\hdots, C_{k_0}) ,  (X_1,\hdots, X_{k_0})\right)_{\text{TV}} \to 0$$ for all fixed $k_0$. Then, the probability 
	that any fixed number of permutations sampled according to $\kappa_n$ invariably generate a transitive subgroup of $S_n$ goes to 0 as $n \to \infty$. 
\end{cor}

This corollary is particularly relevant to the results in \cite{weights1} and \cite{weights2}. They describe measures on $S_n$ that are formed by weighting cycle lengths by parameters $c_j$. The limiting cycle structure has $C_j \sim \Poi(c_j/j)$. Thus, we can create any limiting Poisson cycle structure we like. In particular, those satisfying \thref{cor:lb}.

Another source of alternate measures with limiting Poisson cycle counts comes from \cite{permutons}. This fits into the larger body of work on permutons initiated in \cite{permutons1}. Of particular interest is Mallow's measure (introduced in \cite{original_mallows}). This measure specifies a parameter $q_n>0$ that biases towards more or less inversions in a permutation. A permutation $\pi$ has probability proportional to $q_n^{\text{inv}(\pi)},$
where
$\text{inv}(\pi) := |\{(s, t) | s < t \text{ and } \pi s > \pi t\}|$.

 The recent article (\cite{peled}) begins to characterize the cycle structure of Mallow's measure. 
They find that for $(1-q_n)^{-2} \ll n$ all cycles are on the order of $o(n)$. The high density of small cycles, along with \cite[Theorem 1.4]{permutons}, which proves a limiting Poisson cycle profile, suggests that no finite collection of permutations sampled according to Mallow's measure in this regime will invariable generate $S_n$. Note that in the regime $(1-q_n)^{-2}\gg n$ the cycle counts converge to those of a uniformly random permutation. 
This is a new and exciting area. We are hopeful our result will find more applications as these objects become better understood.

\subsection{Generalizations for the Ewens sampling formula}

In proving \thref{thm:ig} we connect an approximation sumset model (described in Section \ref{sec:overview}) back to fixed-set sizes in random permutations.
This also takes place in \cite{ppr} and \cite{efg}. 
These two articles use a small-cycle limit theorem for uniformly random permutations (\cite[Theorem 1]{arratia}). Namely, that the number of $\ell$-cycles in a uniformly random permutation converges to a Poisson random variable with mean $1/\ell$ that is independent of the other small cycles.

Permutations sampled from $\mu_n$ have an analogous limit theorem (\cite[Theorem 3.2]{arratia2000}). Except now the $\ell$-cycles behave like Poisson random variables with mean $\alpha/\ell$. This lets us use similar ideas to analyze the corresponding sumset model. However, our approach diverges significantly when we connect back to random permutations. Especially in obtaining an upper bound on $m_\alpha$.  
The difference is that both \cite{ppr} and \cite{efg} have access to a large canon of results for uniformly random permutations. We do not. This requires a deep excavation where we extend classical results to $\pi \sim \ESFn$. Along the way we prove many new estimates (see \thref{lem:feller}, \thref{lem:jointcycles}, \thref{lem:2same}, \thref{prop:gcd}, \thref{claim:claim3}, and \thref{lem:odd}) for larger cycle lengths in $\ESFn$ permutations. 


We begin by showing that with high probability the only transitive subgroups containing a permutation $\pi$ sampled with distribution $\ESFn$ are $A_n$ and $S_n$. This is the analogue of what \cite[Theorem 1]{luczak} proves for uniformly random permutations. This was recently studied in more detail by Eberhard, Ford, Green in \cite{efg2}. We plan to explore this vein for $\ESFn$ permutations in future work. For our current purposes, the following result suffices. 

\begin{thm}\thlabel{thm:pyber}
	Let $\pi \sim \ESFn$. 
	Then with probability $1-o(1)$ the only transitive subgroups containing $\pi$ are $A_n$ and $S_n$.
\end{thm}

Pyber and Luczak rely on an important theorem of Bovey regarding primitive subgroups. We need the analogue of it for $\ESFn$ permutations. 
The \emph{minimal degree of} $\langle \pi\rangle = \{\pi^k \colon k \in \mathbb Z\}$ is the minimum number of elements of $\{1,\hdots,n\}$ displaced by some power of $\pi$. 

\cite[Theorem 1]{bovey1980probability} says that if $\pi$ is a uniformly random permutation, then for each $\epsilon>0$ and $0<\beta<1$ we have
$$
  \P[\text{minimal degree of }\langle \pi\rangle > n^\beta  ] <C_{\epsilon,\beta}n^{\epsilon-\beta}.
$$
We establish a weaker analogue of this result.

\begin{thm} \thlabel{thm:bovey}
Let $\pi \sim \ESFn$. For each $0<\beta<1$ and any $\alpha$ it holds that 
	$$\P[\text{minimal degree of }\langle \pi\rangle > n^\beta	] = o(1).$$
\end{thm}
 
Establishing \thref{thm:pyber} also requires a generalization of a classical theorem of  Erd\"os and Tur\'an (see \cite[Theorem V]{erdos}).

\begin{thm} \thlabel{claim:claim2}
Let $\pi \sim \ESFn$ and $\omega(n) \to \infty$. The probability that the largest prime which divides $\prod_{\ell=1}^n \ell^{C_\ell}$ is larger than $n \exp(-\omega(n) \sqrt{\log n})$ is $1-o(1)$.	
\end{thm}

The proof of \thref{thm:pyber} uses both \thref{thm:bovey} and \thref{claim:claim2}.
For all three generalizations we follow a somewhat similar blueprint to their predecessors. However, the previous work often uses special features of uniformly random permutations. 
Counting arguments are heavily employed. In general, these techniques do not translate to $\ESFn$ permutations. Our approach is to use the Feller coupling (see Section \ref{sec:feller}) to directly relate the cycle structure to independent Poisson random variables. In some places this yields more elegant proofs, in others it becomes a bit technical. Keeping in mind that $\ESF(1,n)$ is the uniform measure, this is a nice high-level approach to extending these results.  

\subsection{Overview of proof and result for sumsets} \label{sec:overview}



 
Because a fixed set's size is a sum of cycle lengths of a permutation, invariable generation is related to sumsets formed from random multisets.
The link is developed over a collaborative arc of Arratia, Barbour and  Tavar\'e,  \cite{arratia1992, arratia2000, arratia2013, arratia2016}.  They study the relationship between $\alpha$-logarithmic structures and Poisson random variables. Most relevant for our purposes are the descriptions of the cycle counts $\vec C = (C_1,\hdots,C_n)$ for permutations sampled according to $\mu_n$.

 They prove in \cite[Theorem 3.1, Theorem 3.2]{arratia2000} that, for permutations induced by the strong $\alpha$-logarithmic property, the small cycle counts evolve to be independent Poisson random variables, and the large cycle counts follow Ewens sampling formula. Furthermore, \cite[Lemma 1]{arratia1992} (also discussed in \cite{arratia2016}) shows that the Ewens sampling formula can be cleanly related via the Feller coupling to independent Poi($\alpha/k$) random variables. The payoff is that we can model the sizes of fixed sets in a permutation with independent Poisson random variables.

Let $\mathbf X(\alpha) = (X_1, X_2, \hdots)$ be a sequence of Poisson random variables where $X_j$ has mean $\alpha/j$. The coupling discussed in the previous paragraph allows us to model the sizes of fixed sets in a permutation
  with the random sumset 
 $$\L (\mathbf X(\alpha) ) = \left\{ \sum_{ j \geq 1}  j x_j \colon 0 \leq x_j \leq X_j \right\}.$$
 
Returning to our opening question, we can bound the required number of permutations from above if
no fixed set size is common to all permutations. 
And we can bound the number from below if there is guaranteed to be a common fixed set size. 
Consider independent realizations $\X^{(1)}(\alpha),\ldots,\X^{(m)}(\alpha)$. 
The upper and lower bounds correspond to finding extremal values of $m$ such that
\begin{enumerate}[label = (\roman*)]
  \item $\bigcap_{i=1}^m\L(\mathbf X^{(i)}(\alpha)) = \{0\}  
  \text{ with positive probability}\label{eq:finite}$, and \vspace{.2 cm}
  
  \item $|\bigcap_{i=1}^m\L(\mathbf X^{(i)}(\alpha))| = \infty$ almost surely.$\label{eq:infinite}$
\end{enumerate}
Since common elements among independent sumsets correspond to common fixed-set sizes of independent permutations, invariable generation is analogous to a trivial sumset intersection. On the other hand, an infinite intersection means that common fixed-sets persist among the permutations.
For the case $\alpha =1$, the results of \cite[Theorem 1.6]{ppr} and \cite[Corollary 2.5]{efg} each imply that four sets suffice in \ref{eq:finite}. That \ref{eq:infinite} holds with three sets is due to \cite[Corollary 3.9]{efg}. Here is the analogue for general $\alpha>0$.

\begin{thm} \thlabel{thm:main}
Let $s_\alpha := \inf \{ m \colon \P[ \cap _1^m \L(\mathbf X^{(i)}(\alpha) )  = \{0\} ] >0 \}$
 be the smallest number of i.i.d.\ sumset intersections
so that the resulting set is trivial with positive probability. Let $h(\alpha)$ be as in \eqref{eq:h}, and $m_\alpha$ as in \thref{thm:ig}. At points where $h$ is continuous we have $s_\alpha = h(\alpha),$
and at the discontinuities $h(\alpha) \leq s_{\alpha} \leq h(\alpha)+1.$

\end{thm}

To establish this result for sumsets, we build upon the ideas in \cite{ppr} and \cite{efg} to prove matching lower and upper bounds for $s_\alpha$. These bounds can be found in \thref{prop:upper} and \thref{prop:lower}. Parts of the argument are as simple as swapping 1's for $\alpha$'s, but in others, particularly the lower bound, some care is required. 

A relevant quantity for establishing the upper bound  is
$p_k = p_k(\alpha)$, the probability the element $n$ appears in $\L(\X)$.
Estimating $p_k(1)$ is the major probabilistic hurdle in \cite{ppr}. There are two difficulties. The first is called a lottery effect; certain unlikely events greatly skew $p_k$. This is circumvented by using \emph{quenched} probabilities $\tilde p_k$, which ignore an $o(1)$ portion of the probability space (that $\sum X_j$ and $\sum j X_j$ are uncharacteristically large). With this restriction, it is shown in \cite[Lemma 2.3]{ppr} that $\tilde p_k(1) \leq k^{\log (2) -1}$. Notice that $4 (\log 2 -1) < -1$ and a Borel-Cantelli type argument implies finiteness of the intersection.
 Establishing the quenched formula is the second hurdle. It requires a clever counting and partitioning of $\L(\X)$. 
 We generalize these ideas in \thref{lem:quenched} to obtain the upper bound $$\tilde p_k(\alpha) \leq C_\epsilon k^{- 1 + \alpha \log 2+ \epsilon }$$ for any $\epsilon>0$. 
Taking into account the lottery effect, and ignoring a vanishing portion of the probability space, we can write $\P[k \in \cap_1^m \L(\X^{(i)})] = \prod_1^m \tilde p_k$.
When $m \geq \left\lceil( 1- \alpha \log 2 + \epsilon)^{-1}  \right\rceil$, the product is smaller than $n^{-1}$,
and the probabilities are summable. 
With this, the Borel-Cantelli lemma implies the intersection is almost surely finite, 
and thus it is trivial with positive probability. 

The lower bound in \thref{prop:lower} requires more care. We generalize the framework in \cite[Section 3]{efg} to higher dimensions to incorporate more than three sumsets as well as different values of $\alpha$. Given $I \subseteq \mathbb N$ and $m$ sumsets we consider the set of differences
$$ S(I;\X^{(1)},\ldots,\X^{(m)}):=\{ (n_1 - n_m, n_2 - n_m, \hdots, n_{m-1} - n_m)\colon n_i \in \L(\X^{(i)}) \cap I \}.$$
We use Fourier analysis (as in \cite{efg} and \cite{MT}) to study a smoothed version of the indicator function of this set. In particular, we will work with its Fourier transform $F\colon \m T^{m-1}\to \m C$. After obtaining a pointwise bound on $F$, we integrate around that torus to obtain, 
for sufficiently large $k$ and $I = (k^{1-\beta}, k]$, that $|S(I, M_1,\hdots, M_m)| \gg k^{m-1}$  and lies in the cube $[-ck,ck]^{m-1}$ for some $c>0$. 
The density of these sets is sufficiently high that if we look in a larger interval, with positive probability we can find $j_1, \ldots, j_m \in (k, Dk]$ such that $X_{j_1}^{(1)}, \ldots , X_{j_m}^{(m)} >0$ and
\[
	(j_m - j_1, \ldots, j_m - j_{m-1}) \in S(I; \X^{(1)} , \ldots, \X^{(m)}).
\]
Combining this with the corresponding point in the set of differences then produces a number in $\L(\X^{(1)}) \cap \cdots \cap \L(X^{(m)})$ made solely from values in $(k^{1-\beta}, Dk]$ for some constants $\beta,D >0$. 
By partitioning $\mathbb{N}$ into infinitely many disjoint sets of this form and repeating this argument for each of them, we conclude that the set of intersections is almost surely infinite.

\subsection{Further questions}
There is still much to be done for both components of this paper---random permutations and random sumsets. Recall that \thref{thm:ig} and \thref{thm:main} do not pin down the exact behavior at points of discontinuity of $h$. A glaring question is to characterize $m_{\alpha}$ and $s_{\alpha}$ at these points. We are doubtful our approach generalizes. Characterizing the behavior at the critical values is likely difficult, and will require a new idea. 

\begin{question} \thlabel{q:critical}
What are $m_{\alpha}$ and $s_{\alpha}$ at discontinuity points of $h$?	
\end{question}

For random permutations we would like a more complete characterization than in \thref{cor:lb}. In particular an upper bound. 

\begin{question}
Show that if a permutations has cycle counts that are asymptotically dominated by Poisson random variables with mean $h\inv(2)/k$, then two will invariably generate $S_n$ with positive probability.
\end{question}

\thref{cor:lb} shows that for many permutation measures no finite number will invariably generate $S_n$. Possibly if we let the number of permutations grow with $n$ we will see different behavior.

\begin{question}
	
Let $\alpha > 1 / \log 2$. Prove that there are 
constants $\beta< \beta'$ such that the probability $\beta \log n$  Ewens-$\alpha$ permutations generate $S_n$ is bounded away from 0, while $\beta' \log n$ fail to generate $S_n$ with high probability?
\end{question}

We believe that $\log n$ is the correct order by the following approximation model. When $\alpha > 1 / \log 2$ the fixed set sizes in a random permutation are dense in $[n]$. So, we can model the occurrence of each fixed size by independent Bernoulli random variables with some parameter $p_\alpha >0$. If one performs $N$ independent thinnings of $[n]$ by including each integer $k$ with probability $p_\alpha$, then the probability $k$ belongs to all $N$ subsets is $p_\alpha^N$. Hence, the probability there are no common elements in the $N$-fold intersection is $(1-p_{\alpha}^{N})^n.$ If we choose $N = c \log n$ this will converge to either 0 or a positive number as we increase $c$. The difficulty in answering the question for permutations is we do not have independence for fixing different fixed set sizes. 

We also have a question in the simpler setting of generating $S_n$. Dixon's theorem \cite{dixon} implies that two uniformly random permutations will generate either $A_n$ or $S_n$ with high probability. Does the analogue hold for $\ESFn$ permutations? Because generation relies on more than just cycle structure, this may be a hard question.

\begin{question}
	How many $\ESFn$ permutations are needed to generate one of $A_n$ or $S_n$ with high probability?
\end{question}

\subsection{Notation} \label{sec:notation}
Note that $\alpha>0$ is fixed throughout. As already discussed, we denote cycle counts of a permutation by $\vec C = (C_1,\hdots, C_n)$ with $C_j$ the number of $j$-cycles. Unless otherwise noted, our permutations and cycle vectors come from  an $\ESFn$ distribution. We will let $\vec Y(\alpha)$ denote a vector of independent $\Poi(\alpha/k)$ random variables that are obtained from the Feller coupling in \eqref{eq:feller}. We will also take $\X(\alpha)$ to be a vector of independent $\Poi(\alpha/k)$ random variables. This is done to distinguish the permutation and random sumset settings.  Henceforth we will suppress the $\alpha$ dependence unless there is reason to call attention to it. For an infinite vector, such as $\Y$, we will use the notation $\Y[i,j]$ to denote the sub-vector $(Y_i, Y_{i+1}, \hdots, Y_j)$.  

For $I$ an interval in $\m N$ define the sumset
\[
  \L(I,\X):=\left\{\sum_{j\in I}jx_j\colon 0\leq x_j\leq X_j\right\},
\]
and set $\L(\X):=\L(\m N,\X)$. Given $\X$ define 
\begin{align}f_{\ell,k}(\X)  = \sum_{ \ell < j \leq k } X_j,\qquad g_{\ell,k}(\X) = \sum_{\ell < j \leq k } j X_j. \label{eq:def}
	\end{align}
It will be convenient to abbreviate $f_k := f_{1,k}(\X)$ and $g_k := g_{1,k}(\X)$. Thus, $f_k$ is the number of available summands smaller than $k$, and $g_k$ is the largest sum attainable with them. 


We will use the notation $f\ll g$ and $f(n) = O(g(n))$ interchangeably for the existence of $c$ such that for all large enough $n$ it holds that $f(n) \leq c g(n)$. When the $c$ depends on our choice of $\epsilon$ we will use $f \ll_\epsilon g$, similarly $f \ll_{m,\epsilon}$ denotes dependence of the constant on both $\epsilon$ and the number of intersections $m$. We use the little-$o$ notation $f(n) = o(g(n))$ to mean that $f(n)/g(n) \to 0$. 
Also, $f(n) = \tilde O (g(n))$ means there is $k>0$ such that $f(n) = O( \log ^k(n) g(n)).$
The symbol $f(n) \approx g(n)$ means that $f(n)/g(n) \to c$ for some $c>0$. A sequence of events $E_n$ occurs \emph{with high probability} if $\liminf \P[E_n] =1$. Also $i \mid j$ means that $i$ divides $j$ (i.e.\ there is an integer $k$ with $ik =j$.)

\subsection{Outline of paper}
We will lead off with a proof of \thref{thm:ig}. This comes in two parts: an upper bound from \thref{thm:main}, and a matching lower bound in \thref{thm:main2}. The upper bound in \thref{thm:main} has two distinct components.  We connect random sumsets to $\ESFn$ permutations in Section \ref{sec:feller}. This section includes the proofs of \thref{thm:pyber}, \thref{thm:bovey}, and \thref{claim:claim2}.
The second component is analyzing the random sumset model. We do this in Section \ref{sec:upper_bound} and provide an upper bound on $s_\alpha$ at \thref{prop:upper}.
The lower bound is \thref{thm:main2} and proven in Section \ref{sec:lower_bound}.

\subsection{Acknowledgements} 

We would like to thank Robin Pemantle for several helpful conversations. We are very grateful for Richard Arratia's correspondence. He made us aware of the indispensable Feller coupling. Also, thanks to Clayton Barnes, Emily Dinan, and Jacob Richey for their assistance in the formative stages of this project.

\section{Proof of \thref{thm:ig}} \label{sec:proof}

First we show that $h(\alpha)+1$ is an upper bound for the number of $\ESFn$ permutations to invariably generate $S_n$.

\begin{prop} \thlabel{ESF}
Suppose $\alpha$ is a continuity point of $h$. A collection of $h(\alpha)$ independent permutations sampled according to an $\ESFn$ invariably generates $S_n$ with positive probability as $n\to \infty.$ 	
\end{prop}

\begin{proof}
Let $m = h(\alpha)$.
Call the permutations $\pi^{(1)},\hdots,\pi^{{(m)}}$. 
Condition $\pi^{(1)}$ to be odd. \thref{lem:odd} guarantees this happens with probability bounded away from zero. Define the events
\begin{align*}
	E_{n,k} &= \{\text{$\pi^{(1)},\hdots,\pi^{(m)}$ fix an $\ell$-set for some $k \leq \ell \leq n/2$} \},\\
	F_{n,k} &= \{\text{$\pi^{(2)}$ has no $\ell$-cycles with $\ell \leq k$}\}.
\end{align*}
The probability of a common fixed set size amongst all of the permutations is bounded above by
\begin{align}\P[E_{n,k} \cup F_{n,k}^c] \leq \P[E_{n,k}] + \P[F_{n,k}^c].\label{eq:key_bound}
	\end{align}
Therefore, the $\pi^{(i)}$ will invariably generate a transitive subgroup of $S_n$ if this quantity is less than 1.
%
%

Now $\alpha$ is a continuity point of $h$, so 
\[
	m = h(\alpha) = \lceil (1 - \alpha \log 2)^{-1} ) \rceil > (1 - \alpha \log 2)^{-1}.
\]
It follows that there exists some $\epsilon > 0$ such that 
\[
	m = (1 - (\alpha + \epsilon) \log 2)^{-1}.
\]
Fix $\epsilon'$ such that $0 < \epsilon ' < \epsilon$.
By \thref{lem:pik} for any $k$ and $i = 2, \ldots, m$,
\[
 	\P [ \pi^{(i)} \text{ fixes a $k$-set}] \leq c k^{-\delta},
\] where
\[
	\delta = 1 - (\alpha + \epsilon ') \log 2.
\]
Similarly, by \thref{cor:odd} for any $k$,
\[
	\P [ \pi^{(1)} \text{ fixes a $k$-set}] \leq c k^{-\delta} \cdot \max\left\{\f {\alpha +1}{\alpha}, \f {\alpha+1}{1} \right\}.
\] 
By independence of the $\pi^{(i)}$, we conclude 
$$\P[E_{n,k}] \leq (ck^{-\delta})^{m} \cdot \max\left\{\f {\alpha +1}{\alpha}, \f {\alpha+1}{1} \right\}.$$ 
By construction we have that $\delta m > 1$.
A union bound for $k_0 \leq k \leq n/2$ ensures that, given some $\beta>0$, we can choose $k_0$ such that $\P[E_{n,k_0} ] < \beta$.
Moreover, \cite[Theorem 3.1]{arratia2000} ensures that the cycle counts $\C[1,k_0]$ of $\pi^{(2)}_n$ converge in total variation to $\Y[1,k_0]$. Hence, $$\liminf_{n \to \infty} \P[F_{n,k_0}]>0.$$ Set $\beta$ equal to half this limit infimum. Now, for an appropriate $k_0$ and large enough $n$, it holds that \eqref{eq:key_bound} is strictly less than $1$. 
 %
 This means that with positive probability the collection invariably generates a transitive subgroup of $S_n$. 
  By \thref{thm:pyber} we know this transitive subgroup is with high probability either $A_n$ or $S_n$. Because $\pi^{(1)}$ is odd, it must be $S_n$.
\end{proof}

\begin{remark} \thlabel{rmk:mono}
The argument above implies that $m_\alpha \leq m_{\alpha+\eta}$ for all $\eta >0$, since any appropriate choice of $\epsilon'$ in the latter case will hold for the former. It follows that $h(\alpha)+1$ permutations will invariably generate with positive probability at discontinuity points of $h$.
\end{remark} 
Now we prove a matching lower bound.

\begin{prop} \thlabel{thm:main2}
A collection of $m < h(\alpha)$ permutations sampled independently according to an $ESF(\alpha,n)$ with high probability do not invariably generate $S_n$, or any transitive subgroup of $S_n$.
\end{prop}
\begin{proof}
Let $m< h(\alpha)$, and fix some $\epsilon > 0$.
By \thref{prop:lower} the intersection of the i.i.d.\ sumsets with parameter $\alpha$, $\bigcap_{i = 1}^m \L (\X^{(i)} )$, is almost surely infinite.
It follows that there is some $k_0 = k_0(\epsilon)$ such that if $k \geq k_0$ then $\bigcap_{i = 1}^m \L (\X^{(i)} ) \cap [1,k] \neq \emptyset$ with probability at least $1 - \epsilon /2$.

Fix some $k \geq k_0$ and call the random permutations $\pi^{(1)},\hdots,\pi^{{(m)}}$. 
By \cite[Theorem 3.1]{arratia2000} the cycle counts $\C^{(i)} [1,k]$ of $\pi^{(i)}$ converge in total variation to $\X^{(i)} \cap [1, k]$.
Then there exists some $n_0 = n_0(\epsilon , k)$ such that if $n \geq n_0$, with probability greater than $1 - \epsilon / 2$ we have $\C^{(i)} [1, k] = \X^{(i)} \cap [1, k]$ for all $i = 1, \ldots, m$.
Since 
\[
	\bigcap_{i =1}^m \L (\X^{(i)} ) \cap [1,k] \subset \bigcap_{i=1}^m \L \left(\X^{(i)} \cap [1,k] \right),
\]
with probability at least $1 - \epsilon$ we have that $\pi^{(1)}, \ldots, \pi^{(m)}$ each fix a set of size $\ell$ for some $\ell \leq k$.
As $\epsilon$ was arbitrary, this completes our proof.
\end{proof}

\thref{thm:ig} is a straightforward consequence of the last two propositions.

\begin{proof}[Proof of \thref{thm:ig}]
Together \cite[Theorem 3.1, Theorem 3.2]{arratia2000} guarantee that  the cycle counts of any measure satisfying the \salp~ converge in total variation to the same distribution as an $\ESFn$ permutation. Hence it suffices to consider $\ESFn$ permutations.

By \thref{ESF} we have $m_\alpha \leq h(\alpha)$ at point of continuity of $h$. 
At points of discontinuity, we use \thref{rmk:mono} to deduce that $m_\alpha \leq m_{\alpha + \eta} \leq h(\alpha)+1$.
Then, by \thref{thm:main2} we have that $h(\alpha) \leq m_\alpha$ for all $\alpha$, completing our proof. 
\end{proof}

\section{The upper bound for $m_\alpha$}

The main goal of this section is to prove that $m_\alpha \leq s_\alpha$. Our primary tool for relating Poisson sumsets to permutations is the Feller coupling. This is described at the onset of Section \ref{sec:feller}. We then use it in that section to prove several estimates on the cycle structure of $\ESFn$ permutations. With these we prove \thref{thm:pyber}, \thref{thm:bovey}, and \thref{claim:claim2} in Section \ref{sec:bovey_and_pyber}. We put all of this together and establish \thref{ESF} in Section \ref{sec:ESF}.

\subsection{The Feller Coupling} \label{sec:feller}

The Ewens sampling formula can be obtained via an elegant coupling attributed to Feller \cite[p.\ 815]{feller1945fundamental} for uniform permutations. The articles \cite{arratia1992,arratia2016} have nice descriptions. We start with a sequence $\xi_1,\xi_2,\hdots$ of mixed Bernoulli random variables with $\P[\xi_i = 1] = \f{\alpha}{\alpha+i-1}$ and $\P[\xi_i = 0] = \f{i-1}{\alpha+i -1}.$ Define an $\ell$-spacing to be any occurrence of $\ell-1$ zeros enclosed by a 1 on the left and right (e.g.\ $1,0,0,0,1$ is a $4$ spacing). Remarkably, the total number of $\ell$-spacings, $Y_\ell$, is distributed as $\Poi(\alpha/\ell)$. Moreover, the collection $Y_1,Y_2,\hdots$ are independent (see \cite{arratia2016, arratia1992}). 

The counts of spacings in $\xi_1,\hdots,\xi_n,1$ generate a partition of $n$, which can be filled in uniformly randomly to form a permutation with $C_\ell$ cycles of size $\ell$, where $C_\ell$ are sampled according to an $ESF(\alpha,n)$ distribution. Explicitly,
%
%
\begin{align} C_\ell &=\text{the number of $\ell$-spacings in $\xi_1,\hdots,\xi_n,1$} \nonumber \\
					&=\xi_{n-\ell +1}(1-\xi_{n-\ell+2})\cdots (1- \xi_{n}) + \sum_{k=1}^{n- \ell} \xi_k(1-\xi_{k+1})\cdots(1- \xi_{k+\ell-1}) \xi_{k+\ell} \label{eq:C_formula}.
\end{align}
This gives an intuition for why $ESF(1,n)$ corresponds to a uniformly random permutation. Indeed, we can inductively construct such a permutation by letting $\xi_i$ indicate the decision to complete a cycle, when there is an $i$-way choice for the next element. 

Letting $R_n$ be the index of the rightmost 1 in $\xi_1,\hdots, \xi_n$ and $J_n = n+1 - R_n$ it follows from the previous discussion that
\begin{align}
C_\ell \leq Y_\ell + \ind{J_n = \ell} \label{eq:feller}.\end{align}
This says that the cycle counts 
can be obtained from independent Poisson random variables through a random number of deletions, and possibly one insertion. We require a few estimates on the behavior of these perturbations.

\begin{lemma} \thlabel{lem:feller} Let $J_n$ be as defined above, and $D_n = \sum_{\ell = 1}^n Y_\ell - \ind{J_n = \ell} 
- C_\ell$ be the number of deletions. The following four inequalities hold:
  
			\begin{align}
				\P[J_n = \ell] &\leq \f{\alpha}{n -\ell},  \qquad\qquad \quad \:\ell < n \label{eq:J1}.	
		\\\nonumber \\
				\P[J_n = \ell] &\leq c n^{-\alpha(1-\gamma)},\qquad  \;n-n^{\gamma} < \ell <n, \quad 0<\gamma <1, \textnormal{ for some $c=c(\gamma) > 0$}.\label{eq:J2}
		\\\nonumber \\
				\P[J_n = \ell] &= O(n^{- \alpha/(1+ \alpha)}), \qquad \ell < n\label{eq:J3}.
		\\\nonumber \\
			\E D_n & = O(1).\label{eq:D} 
			\end{align}

\end{lemma}

\begin{proof}
The inequality in \eqref{eq:J1} follows from the formula given in \cite[Remark following (17)]{arratia1992}, which says $\P[J_n = \ell] =  \ell \E C^{(n)}_\ell / n.$ The coupling at \eqref{eq:feller} ensures that $\P[J_n = \ell] \leq  (\alpha + \ell \P[J_n = \ell])/n$. Solving for $\P[J_n = \ell]$ gives \eqref{eq:J1}.


Fix some $\gamma \in (0,1)$. For the case $\ell>n-n^{\gamma}$ in \eqref{eq:J2} we use that when $J_n=\ell$ we must have $\xi_{n-\ell +1} =1$ and $\xi_{k}=0$ for all larger index terms up to $\xi_n$. The probability of this can be computed explicitly as 
\begin{align*}
		\P[J_n = \ell]&= \frac{\alpha}{\alpha+n-\ell}\prod_{n-\ell+2}^{n} \frac{i-1}{\alpha+i-1} 
					  \leq \prod_{n-\ell+1}^{n} \frac{i-1}{\alpha+i-1}.
\end{align*}
We estimate the product by converting it to a sum
\[
\prod_{n-\ell+1}^{n} \frac{i-1}{\alpha+i-1}=\exp\biggl(-\sum_{n-\ell+1}^{n} \log \frac{\alpha+i-1}{i-1} \biggr)\leq  \exp\biggl(-\int_{n-\ell+1}^{n} \log \frac{\alpha+x-1}{x-1} dx\biggr).
\]
The last inequality follows from the fact that $\log\frac{\alpha+x-1}{x-1}$ is a decreasing function, and the sum can be viewed as a right-sided Riemann (over-)approximation. This integral has a closed form
\begin{align}
-\int_{n-\ell+1}^{n} \log\left(\frac{\alpha+x-1}{x-1}\right)dx = -\log \left(\frac{(\alpha+n-1)^{\alpha+n-1}}{(n-1)^{n-1}} \right)+ \log \left(\frac{(\alpha+n-\ell)^{\alpha+n-\ell}}{(n-\ell)^{n-\ell}} \right). \nonumber
\end{align}
Hence, exponentiating the above line gives
\begin{align}
\prod_{n-\ell+1}^{n} \frac{i-1}{\alpha+i-1} & \leq \frac{(n-1)^{n-1}}{(\alpha+n-1)^{\alpha+n-1}} \times \frac{(\alpha+n-\ell)^{\alpha+n-\ell}}{(n-\ell-1)^{n-\ell}}. \label{eq:hot_mess}
\end{align}
The leftmost product is less than $n^{-\alpha}$. The rightmost product can be written as
\begin{align}(\alpha + n- \ell)^\alpha \left( \f{n - \ell+\alpha}{n-\ell -1} \right)^{n- \ell}= (\alpha + n- \ell)^\alpha \left( \f{1+\f{\alpha}{n-\ell}}{1-\f{1}{n-\ell}} \right)^{n- \ell}. \label{eq:hot_mess2}	
\end{align}
The rightmost product in \eqref{eq:hot_mess2} is asymptotic to $e^{\alpha}/e^{-1}$ as $n- \ell \to \infty$, and thus may be bounded by some constant for all values of $n-\ell$. 
Using the hypothesis that $n-\ell \leq n^{\gamma}$, it follows that \eqref{eq:hot_mess2} is universally bounded by $c n^{\alpha\gamma }$ for some constant $c$. 
We conclude that \eqref{eq:hot_mess} is less than $c n^{-\alpha + \alpha \gamma} = c n ^{ -\alpha ( 1- \gamma)}.$ This establishes \eqref{eq:J2}.

%
%
%

To obtain the universal bound at \eqref{eq:J3}, observe when $\ell \leq n - n^{\alpha/(1+\alpha)}$ we can input this into \eqref{eq:J1} to obtain
$$\P[J_n = \ell]  \leq \f{ \alpha}{n - (n - n^{\alpha/(1+\alpha)})} = O(n^{-\alpha/(1+\alpha)}).$$
And, for $\ell > n - n^{\alpha}$ we substitute $\gamma = \alpha/(1+\alpha)$ into \eqref{eq:J2} to obtain the same asymptotic inequality: $O(n^{-\alpha/(1+ \alpha)})$.

Lastly, \eqref{eq:D} is  proven in \cite[Theorem 2]{arratia1992}.
\end{proof}

We will also require a bound on the joint probability that two cycles occur simultaneously. There will be some poly-log terms that, besides being a minor nuisance, do not effect the tack of our proof. Recall the $\tilde O$ notation that ignores logarithmic contribution. It is described carefully in Section \ref{sec:notation}. 


\begin{lemma}\thlabel{lem:jointcycles}
For any $i<j \leq n$ it holds that
$$\P[C_iC_j >0] = \tilde O(i^{-1} j^{- \alpha/(1+ \alpha)} ).$$
\end{lemma}
\begin{proof}
We start by using the formula at \eqref{eq:C_formula} to bound $C_i$ with the $\xi_k$,
$$C_i \leq \xi_{n-i +1}(1-\xi_{n-i+2})\cdots (1- \xi_{n}) + \sum_{k=1}^{n-i} \xi_k \xi_{k+i}.$$
Thus,
\begin{align*}C_iC_j &\leq \Bigl(\xi_{n-i +1}(1-\xi_{n-i+2})\cdots (1- \xi_{n})+ \sum_{k=1}^{n-i} \xi_k \xi_{k+i} \Bigr) \\
& \qquad \qquad \times \Bigl(\xi_{n-j +1}(1-\xi_{n-j+2})\cdots (1- \xi_{n})  +\sum_{\ell=1}^{n-j} \xi_\ell \xi_{\ell+j} \Bigr)\\
&=(x+y)(z+w).
\end{align*}
The last line is just labeling the terms in the product immediately above in the natural way (i.e. 
$y= \sum_{k=1}^{n-i} \xi_k \xi_{k+i}$). This is so we can bound $\E[xz+ xw+ yw + yz]$ in an organized fashion, one term at a time. 

\begin{description}
	\item[Term $xz$]

		First, consider $\xi_{n-i +1}(1-\xi_{n-i+2})\cdots (1- \xi_{n})\xi_{n-j +1}(1-\xi_{n-j+2})\cdots (1- \xi_{n}).$
		When $i<j$ this product contains the term $\xi_{n-i+1}(1- \xi_{n-i+1}) \equiv 0$, thus it is always zero. We conclude that 
		\begin{align}\E[xz] \equiv 0.\label{eq:xz}\end{align}

	\item[Term $xw$] 

		Next, we consider the cross-term $\xi_{n-i +1}(1-\xi_{n-i+2})\cdots (1- \xi_{n})\sum_{\ell =1}^{n-j} \xi_\ell \xi_{\ell +j}.$ Many of the summands are zero. In fact, the overlapping terms with $n-j-i+1<\ell \leq n-j$ are identically zero. So, for the sake of obtaining disjoint terms, we consider the smaller sum
		\begin{align*}
		\xi_{n-i +1}(1-\xi_{n-i+2})\cdots (1- \xi_{n})  \sum_{\ell =1}^{n-j-i+1}	\xi_\ell \xi_{\ell +j}.
		\end{align*}
		Aside from the $\ell = n-j-i+1$ term, all the indices are all disjoint. 
		Ignoring this term for now, the expected value is
%
		\begin{align}
		\f{\alpha}{\alpha + n-i} \f{ n-i+1}{\alpha+ n - i +1} \cdots    \f{n-1}{\alpha+n-1}  \sum_{\ell =1}^{n-j-i}	\f{\alpha^2}{(\alpha + \ell-1)(\alpha + \ell -1 +j)}. \label{eq:sum}
		\end{align}

		Now $i$ and $j$ correspond to cycle lengths, so we must have $i + j \leq n$ to have a nonzero expectation at all. 
		Since $i<j$ we conclude that $i<n/2$. 
		A simple bound on the product outside the sum is to plug this value into the first term and ignore the rest. This gives the bound
		$$\f{\alpha}{\alpha + n-i} \f{ n-i+1}{\alpha+ n - i +1} \cdots   \f{n-1}{\alpha+n-1} < \f{\alpha}{n - (n/2)+1} < \f{2\alpha}{n}.$$
		Ignoring the product outside of it, the sum can be compared to the integral $\int_1^n \f{1}{x(x+j)}dx$ to prove it is $O(\log(j)/j)$. Thus the expression in \eqref{eq:sum} is $\tilde O(1/nj)$.
		
		Rather than using the rough bound from above, the full $\ell = n-j-i+1$ term is
%
		\[
			\xi_{n-j-i+1}(1- \xi_{n-j-i+2}) \cdots(1-\xi_{n-i}) \xi_{n-i+1}^2 (1 - \xi_{n-i+2}) \cdots (1- \xi_n) .
		\]
		The expectation of this term is
		\[
			\f{\alpha}{\alpha+n-j-i} \left(\prod_{k= n-j-i+2}^{n-i} \f{ k -1}{\alpha + k -1} \right) \f{\alpha}{\alpha+n-i} \left( \prod_{k = n - i+2}^{n} \f{k-1}{\alpha+k-1} \right).
		\]
		But this is just
		\begin{align*}
			\f{\alpha}{\alpha + n -i} \f{\alpha+n-i}{n-i} \P [ J_n = j+i] &= \f{\alpha}{n-i} \P [ J_n = j+i] \\
			&\leq \f{2\alpha}{n}\P [ J_n = j+i].
		\end{align*}
		Then by \thref{lem:feller} \eqref{eq:J3} we have the bound
		\[
			\P [ J_n = j+i] = O(n^{-\alpha/(1+\alpha)} ).
		\]
		Thus this term is $ O(i^{-1} j^{-\alpha/(1+\alpha)})$, and
		\begin{align}
\E[xw]=  		\tilde O(i^{-1} j^{-\alpha/(1+\alpha)}). \label{eq:xw}
		\end{align}

\item[Term $yw$] Because $i<j$ the term $\xi_{n-j +1}(1-\xi_{n-j+2})\cdots (1- \xi_{n})\sum_{\ell =1}^{n-i} \xi_\ell \xi_{\ell +i}$  behaves differently than the case just prior. In fact, it is the largest order term amongst the four. The first step is the same as before though. A few of the terms are identically zero. So, we can write $yw$ as
$$\xi_{n-j +1}(1-\xi_{n-j+2})\cdots (1- \xi_{n})\sum_{\ell =1}^{n-i-j+1} \xi_\ell \xi_{\ell +i}.$$
Except for when $\ell = n-i-j+1$, all of the indices now disjoint.
When we take expectation of the terms with disjoint indices we obtain
$$\f{\alpha}{\alpha + n-j} \f{ n-j+1}{\alpha+ n - j +1} \cdots    \f{n-1}{\alpha+n-1}  \sum_{\ell =1}^{n-j-i}	\f{\alpha^2}{(\alpha + \ell -1) (\alpha + \ell+i -1)}.$$
Notice that $j$ could be quite large. The best bound we have for the product outside the sum comes from \eqref{eq:J3}.
 This implies it is $O(n^{-\alpha/(1+\alpha)})$.
As with \eqref{eq:sum}, the sum term is $O(\log i /i)$. This yields that the whole expression is $\tilde O(i^{-1} j^{-\alpha/(1+\alpha)})$.

As before, the full $\ell = n - i -j +1$ term is 
\[
	\xi_{n-i-j+1} (1- \xi_{n-j-i+2}) \cdots(1-\xi_{n-j}) \xi_{n-j+1}^2 (1 - \xi_{n-i+2}) \cdots (1- \xi_n) .
\]
The expectation of this term is 
	\[			
		\f{\alpha}{\alpha+n-j-i} \left(\prod_{k= n-j-i+2}^{n-j} \f{ k -1}{\alpha + k -1} \right) \f{\alpha}{\alpha+n-i} \left( \prod_{k = n - j+2}^{n} \f{k-1}{\alpha+k-1} \right).
	\]
This simplies to 
\[
	\f{\alpha}{\alpha + n -j} \f{\alpha+n-j}{n-j} \P [ J_n = j+i] = \f{\alpha}{n-j} \P [ J_n = j+i].
\]
By construction $i \leq n-j$, and again by \thref{lem:feller} \eqref{eq:J3}, we have that $\P[J_n = j+i] = O(n^{-\alpha/(1+\alpha)})$.
Thus this term has expectation that is $O(i^{-1} j^{-\alpha/(1+\alpha)})$. It follows that
\begin{align}\E[yw] = \tilde O( i^{-1}n^{-\alpha/(1 + \alpha)}).\label{eq:yw}\end{align}
\item[Term $yz$] The last step is to bound the product 
\begin{align}
\Bigl(\sum_{k=1}^{n-i} \xi_k \xi_{k+i} \Bigr)\Bigl( \sum_{\ell=1}^{n-j} \xi_\ell \xi_{\ell+j} \Bigr) = \sum_{k,\ell=1}^{n-i,n-j}\xi_k \xi_{k+i} \xi_\ell \xi_{\ell+j}.	\label{eq:cross_term}
\end{align}
Call the index set $\Lambda = \{(k ,\ell) \colon k \in [1,\hdots, n-i], \ell \in [1,\hdots,n-j]\}$. As $i <j$, we can divide $\Lambda$ into  disjoint sets 
\begin{align*}
	\Lambda_0 &= \{k, k+i,  \ell, \text{ and } \ell +j \text{ are all distinct}\},\\
	\Lambda_1 &= \{k= \ell\},\\
	\Lambda_2 &= \{ k= \ell+j\},\\
	\Lambda_3 &= \{k+i= \ell\},\\
	\Lambda_4 &= \{k+i= \ell+j\}.
\end{align*}
This yields the following upper bound on \eqref{eq:cross_term}:
\begin{align*}	
	\sum_{\Lambda_0} \xi_k \xi_{k+i}\xi_\ell \xi_{\ell +j} +
	\sum_{\Lambda_1} \xi_k \xi_{k+i}\xi_{\ell +j} +
	\sum_{\Lambda_2} \xi_k \xi_{k+i}\xi_{\ell}+ 
	\sum_{\Lambda_3} \xi_k \xi_{k+i}\xi_{\ell+j} +
	\sum_{\Lambda_4} \xi_k \xi_{k+i}\xi_{\ell}.
\end{align*}
The payoff of this partition is that each product of $\xi$ above is of distinct, independent terms. This lets us tidily compute the expectation of each.
For $(k,\ell) \in \Lambda_0$ we have the $k,k+i,\ell, \ell+j$ are all distinct. This gives
$$\E[\xi_k\xi_{k+i}\xi_{\ell} \xi_{\ell+j}] \leq  \f A  {k(k+i)\ell(\ell+j)}$$
for some constant $A$ depending only on $\alpha$. We can take the expectation of the entire sum
\begin{align*}\E \sum_{\Lambda_0} \xi_{k} \xi_{k+i} \xi_{\ell}\xi_{\ell +j} \leq \sum_{k=1,\ell=1}^{n} \E [\xi_{k} \xi_{k+i} \xi_{k +j}] \leq \sum_{k=1,\ell=1}^{n} \f{A}{k (k+i) \ell (\ell+j) } = \tilde  O( 1/ij).
\end{align*}

We can parametrize $\Lambda_1$ by $k$ to write it as $\Lambda_1 = \{ (k,k) \colon k=1,\hdots, n-j\}$, this gives the bound
\begin{align}\E \sum_{\Lambda_1} \xi_{k} \xi_{k+i} \xi_{\ell +j} \leq \sum_{k=1}^{n} \E [\xi_{k} \xi_{k+i} \xi_{k +j}] \leq \sum_{k=1}^{n} \f{A}{k (k+i) (k+j) } =  \tilde  O( 1/ij).\nonumber
\end{align}
We can similarly parametrize the sums over $\Lambda_2,\Lambda_3$ and $\Lambda_4$ to obtain 
\begin{align} \E[yz] =  \tilde  O(1/ij) \label{eq:yz} 
\end{align}
\end{description}

All of the expressions at \eqref{eq:xz}, \eqref{eq:xw}, \eqref{eq:yw}, and \eqref{eq:yz} are $\tilde O(i^{-1}j^{-\alpha/(1+\alpha)})$. Markov's inequality yields
\begin{align}
\P[C_iC_j > 0] = \tilde O(i^{-1}j^{-\alpha/(1+\alpha)}).
\label{eq:PC_iC_j}
\end{align}
This proves the claimed inequality. 
\end{proof}

We also need a bound on the probability that cycles of the same size occur.

\begin{lemma} \thlabel{lem:2same}
For any $i \leq n$ it holds that
$$\P[C_i \geq 2] = \tilde O(i^{-2} + n^{-1 - \alpha/(1+\alpha)}).$$	
\end{lemma}

\begin{proof}

We start by bounding $\P[C_i >2]$. By \eqref{eq:feller} we have
\begin{equation}
	\P[Y_i + \ind{J_n = i} > 2] \leq \P[Y_i >1] = O(1/i^2). \label{eq:term1}
\end{equation}
It remains to bound $\P[C_i = 2].$ Conditioning on the value of $J_n$, we obtain the decomposition
\begin{align}
\P[C_i = 2 ] &= \P[C_i =2 \mid J_n = i]\P[J_n = i] + \P[C_i = 2\mid J_n \neq i]\P[J_n \neq i]  \label{eq:term2}.
\end{align}
First we estimate the first summand on the right hand side of \eqref{eq:term2}. Because two cycles larger than $n/2$ cannot occur, we must have $i \leq n/2$. 
Then by \thref{lem:feller} \eqref{eq:J1}, $\P[J_n = i] \leq 2\alpha/n.$ When $J_n=i$, it ensures that the sequence $\xi_1,\hdots,\xi_n$ has $\xi_{n-i+1}=1$ and the larger index $\xi_k$ are zero. The smaller index $\xi_i$ are independent of this event. The value of $C_i$ conditioned on $J_n=i$ is thus one more than the number of $i$-spacings in the sequence
$$\xi_1,\hdots,\xi_{n-i},1,$$
with $\xi_1,\hdots \xi_{n-i}$ independent and distributed $ESF(\alpha,n-i)$.
Call the number of $i$-spacings in the above sequence $C_i^-$. So, conditional on $J_n =i$ we have $C_i = C_i^- + 1$. By writing explicitly what it means for an $i$-spacing to arise, we can bound $C_i^-$ in terms of the $\xi_k$:
$$C_i^- \leq \xi_{n-2i+1}(1- \xi_{n-2i +2})\cdots(1- \xi_{n-i}) + \sum_{k=1}^{n-2i} \xi_k \xi_{k+i} .$$
Looking to apply Markov's inequality, we take the expectation of the above line and obtain
$$\E C_i^- \leq \f{\alpha}{\alpha + n-2i} \f{ n-2i+1}{\alpha+ n -2 i +1} \cdots   \f{n-i-1}{\alpha+n-i-1} + \sum_{k =1}^{n-2i}	\f{\alpha^2}{(\alpha + k -1)(\alpha + k -1 +i)}.$$
The product above is equivalent to $\P [J_{n-i} = i]$, then we use \thref{lem:feller} \eqref{eq:J3} and the fact that $i<n/2$ to bound it by $\tilde O( (n-i)^{-\alpha/(1+\alpha)}) = \tilde O (n^{-\alpha/(1+\alpha)})$. 
The sum, as we have seen from \eqref{eq:sum}, has order $\tilde O(1/i)$.
Recalling that \eqref{eq:J1} implies $\P[J_n = i] = O(1/n)$, it follows that
\begin{align}
\P[C_i =2 \mid J_n = i]\P[J_n = i] = 	\tilde O(n^{-1 - \alpha/(1+\alpha)} + 1/in) = \tilde O(i^{-2} + n^{-1 - \alpha/(1+\alpha)}). \label{eq:term1bound}
\end{align}

Now we estimate the second summand on the right hand side of \eqref{eq:term2}. Under the event $\{J_n \neq i \}$, for $C_i$ to equal 2, it is necessary that two $i$-spacings occur in the infinite sequence $\xi_1,\xi_2,\hdots.$ The count of $i$-spacings is distributed as $Y_i$ from \eqref{eq:feller}. Therefore,
\begin{align}
	\P[C_i =2 \text{ and } J_n \neq i ] \leq \P[Y_i \geq 2] = O(1/i^2). \label{eq:term2bound}
\end{align}
Then \eqref{eq:term1}, \eqref{eq:term1bound}, and \eqref{eq:term2bound} imply that 
\begin{align}\P[C_i\geq 2] = \tilde O(i^{-2} + n^{-1 - \alpha/(1+\alpha)}) \label{eq:C_i>=2}\end{align}
as desired.
\end{proof}
Using the previous two lemmas we may prove the following.
 
\begin{prop}\thlabel{prop:gcd}
Let $\pi\sim ESF(\alpha,n)$ and $a = 1 - \f{\alpha}{4 \alpha +3}$. The probability that $\pi$ has two cycles whose lengths have a common divisor larger than $n^{a}$ is $o(1)$. 
\end{prop}

\begin{proof}
Call $E_a$ the event that $\pi$ has two cycles with common divisor larger than $n^a$.
We utilize a union bound, then apply  \thref{lem:jointcycles} and \thref{lem:2same} to write
\begin{align*}
\P[E_a] &\leq \sum_{d=\lceil n^{a} \rceil}^{n} \biggl( \sum_{1 \leq i < j}^{\lfloor n^{1-a} \rfloor }   \P[C_{id}C_{jd} >0] + \sum_{i=1}^{\lfloor n^{1-a} \rfloor } \P[C_{id} \geq 2 ] \biggr)\\
	&= \tilde O  \biggl( \sum_{d=\lceil n^{a} \rceil}^{n} \sum_{1 \leq i < j}^{\lfloor n^{1-a} \rfloor }  (id)^{-1} (jd)^{- \alpha/(1+ \alpha)} + \sum_{d=\lceil n^{a} \rceil}^{n}\sum_{i=1}^{\lfloor n^{1-a} \rfloor }\left( \f{1}{(id)^2}  + n^{-1-\alpha/(1+\alpha)} \right) \biggr).
\end{align*}
Since $d \geq n^a$ the above line is
\[
\tilde O \Bigl( n^1 n^{2(1-a)} n^{- a  -a\alpha/(1+\alpha) } + n^1n^{1-a} n^{-2a} + n^{1-a- \alpha/(1+\alpha)} \Bigr).
\]
We can simplify the exponents in each term to
$$3 -3a - a\alpha/(1+\alpha), \quad  \quad 2-3a, \quad \text{ and } \quad 1-a- \alpha/(1+\alpha),$$
respectively. 
Some algebra shows that all are negative as long as 
$a> 1 - \f{\alpha} {4\alpha +3}.$
For such $a$ we have $\P[E_a] = o(1)$.
\end{proof}

\subsection{Proofs of \thref{thm:pyber}, \thref{thm:bovey}, and \thref{claim:claim2}} \label{sec:bovey_and_pyber}
We now have enough to establish our bound on the minimal degree of an $\ESFn$ permutation.

\begin{lemma} \thlabel{lem:pyberpt1}
	Let $\pi \sim \ESFn$. For each $0 < \beta < 1$, there exists $\lambda \in (0,1)$ such that if we define the set
	\[
		Q_n = \{d < n^{\beta}\colon \text{there exists a prime $p>n^{\lambda \beta}$ with $p \mid d$}\},
	\]
	then
 	\begin{align}
 		\P[\exists C_d =1 \text{ with $d \in Q_n$}] = 1 - o(1) \label{eq:C_d=1} . 
 	\end{align}
\end{lemma}
\begin{proof}
We start with \cite[Theorem 3.1]{arratia2000} which implies
\begin{align}
\sum_{d \in {Q_n}} C_d \overset{TV} \to \Poi(\kappa_n),  \label{eq:Q_n}
\end{align}
where
\[
	\kappa_n = \sum_{d\in Q_n} \frac{\alpha}{d} .
\]

Borrowing an argument from \cite[p. 50]{bovey1980probability} we have
\begin{align*}
	\sum_{d \in Q_n} \f 1 d & = \sum_{d < n^\beta} \f 1 d - \sum_{\substack{d < n^\beta \\ p \mid d \Rightarrow p \leq n^{\lambda \beta}}}  	\f 1 d \\
			&\geq \sum_{d < n^\beta} \f 1 d - \prod_{p \leq n^{\lambda \beta}} ( 1- 1/p)^{-1} \\
			&= \beta \log n - e^{\gamma^*} \lambda \beta \log n + O(1)
\end{align*}
where $\gamma^*$ is Euler's constant. When $\lambda$ is such that $e^{\gamma^*} \lambda \beta < \beta /2$ we have $\lambda_n > \f \beta 2 \log n$ for all $n$. It follows from \eqref{eq:Q_n} and standard estimates on the concentration of a Poisson random variable that $$\P[\sum_{d \in Q_n} C_d > (\beta /4) \log n] \to \P[\Poi(\lambda_n) > (\beta /4) \log n] = 1 - o(1).$$ 
To obtain \eqref{eq:C_d=1}, observe that the convergence statement in \cite[Theorem 3.1]{arratia2000} implies that $\P[C_d >1] \to \P[Y_d >1] = O(1/d^2).$ Here $Y_d$ is a Poisson random variable as in \eqref{eq:feller}. The Borel-Cantelli lemma implies that only finitely many $Y_d$ are larger than $1$. Hence $\sum_{Y_i >1} Y_i$ is almost surely bounded.  
We have seen that, for $n$ large, the sum of $C_d$ with $d \in Q_n$ is at least $\Poi(\f \beta 4 \log n)$ with probability $1-o(1)$. This diverges, and so one or more of the positive $C_d$ must be equal to one. 
Our claim follows. 
\end{proof}

For a random permutation $\pi \sim \ESFn$, we define $\Phi(\pi) = \prod_{\ell = 1}^n \ell^{C_\ell}$. To prove \thref{thm:bovey}, it suffices to show that with high probability $\pi$ has a cycle of length $d$ such that
\begin{align}
	d & \leq n^{\beta}, \label{bovey:1} \\
	d^2 &  \nmid \Phi(\pi) \label{bovey:2}.
\end{align}
To that end we establish one more lemma.

\begin{lemma} \thlabel{lem:pyberpt2}
	Let $\pi$, $\beta$, $\lambda$, and $Q_n$ be as in \thref{lem:pyberpt1}, then
	\begin{align}
		\P[ \text{$\exists d \in Q_n$ such that $d^2 \nmid \Phi(\pi)$ and $C_d = 1$} ] = 1- o(1). \label{eq:goal}	
	\end{align}
\end{lemma}
\begin{proof}
	Call the event in \eqref{eq:goal} $F_n$. Conditioning on the event that $C_d=1$ for some $d \in Q_n$,
$$\P[F_n] \geq \P[F_n \mid \exists C_d =1, d \in Q_n] \P[\exists C_d =1, d \in Q_n].$$
By \thref{lem:pyberpt1} we have that $\P[\exists C_d = 1, d \in Q_n] = 1 - o(1)$, so it suffices to show that 
$$\P[F_n \mid \exists C_d =1, d \in Q_n] = 1- o(1).$$

Fix $d\in Q_n$ with $C_d=1$, we will prove
\[
\P[d^2\mid \Phi(\pi)]=o(1).
\]
From our construction of $Q_n$, there exists a prime $p>n^{\lambda\beta}$ such that
\[
\{d^2\mid \Phi(\pi)\}\subset \{p\mid \f {\Phi(\pi)} d\}.
\]
For such prime number $p$, the event on the right hand side above satisfies
\[
\{p\mid \f {\Phi(\pi)} d\}\subset \{ \sum_{i=1}^{n/p} C_{ip}>1 \}
\]

Call the last event $G_n$. 
The Feller coupling at \eqref{eq:feller} allows us to bound the probability by
\begin{align}
\P [ G_n] \leq \P\Bigl[ \sum_{i=1}^{\lceil n  /d \rceil } Y_{ip} + \ind{J_n = ip} >1  \Bigr].
\end{align}
This, in turn, is bounded by
$$\P\Bigl[ \Poi\bigl( \sum_{i=1}^{\lceil n/p \rceil} \f \alpha {ip} \bigr) > 0 \Bigr] +  \P\Bigl[ \sum_{i=1}^{\lceil n/p \rceil}  \ind{J_n = ip} >0  \Bigr].
$$
The mean of the above Poisson random variable is asymptotically bounded by $\alpha \log(n/p)/p ,$ which is $o(1)$ since $p \geq n^{\lambda\beta}$. This ensures that the left summand is $o(1)$. To bound the right summand we write
\[
\P\Bigl[ \sum_{i=1}^{\lceil n/p \rceil}  \ind{J_n = ip} >0  \Bigr]\leq \sum_{ip\leq n-n^{\gamma}}  \P\left[{J_n = ip}  \right]+ \sum_{ip>n-n^{\gamma}}  \P\left[{J_n = ip}   \right].
\]
We begin by bounding the first summand of the right hand side. By \thref{lem:feller} \eqref{eq:J1},
\begin{align*}
	\sum_{ip\leq n-n^{\gamma}}  \P\left[{J_n = ip}  \right]
	&\leq \sum_{i = 1}^{\lfloor (n- n^{\gamma})/p \rfloor} \frac{ \alpha}{n- ip} \\
	&\leq \int_1^{\frac{n-n^\gamma}{p} + 1} \frac{\alpha}{n - p x} dx \\
	&= O \left( \frac{\log n}{p} \right) \\
	&= o(1).
\end{align*}
The last line follows since $p \geq n^{\lambda \beta}$. For the second summand, from \thref{lem:feller} \eqref{eq:J2} there exists some constant $c>0$ such that
\begin{align*}
	\sum_{ip>n-n^{\gamma}}  \P\left[{J_n = ip}   \right]
		&\leq \sum_{ip>n-n^{\gamma}} c n^{-\alpha(1-\gamma)} \\
		&\leq c n^{\gamma - \lambda \beta} n^{-\alpha(1-\gamma)}.
\end{align*}
The last line follows as there are at most $n^{\gamma - \lambda \beta}$ different values of $i$ such that $n-n^\gamma < ip \leq n$.

We therefore have that
\[ \P [G_n] = O \left( n^{\gamma - \lambda \beta -\alpha(1-\gamma)} \right).
\]
An easy calculation confirms that $\gamma$ satisfying
\[
	\gamma < \frac{\alpha + \lambda \beta}{1+\alpha}
\]
is sufficient to ensure
\[
 	\P[G_n] = o(1).
\]
Choosing a suitable $\gamma$, it follows that 
\[
	\P[F_n] = 1 - o(1),
\]
completing our proof.

\end{proof}

\begin{proof}[Proof of \thref{thm:bovey}]
By \thref{lem:pyberpt2} with probability $1-o(1)$ there exists a cycle length of length $d$ satisfying \eqref{bovey:1} and \eqref{bovey:2}.
%
%
When these conditions are met, for $K = \Phi(\pi)/d$, we have $\pi^K$ displaces $d \leq n^{\beta}$ elements. Thus, with high probability $\pi$ has minimal degree no more than $n^{\beta}$. This completes our proof.
\end{proof}

%

Note that when  $\beta=1/2$ the above holds for all $\alpha>0$. This will be essential for the proof of \thref{thm:pyber}.
We use the previous result to extend \cite[Theorem 1]{luczak} from uniform permutations to $ESF(\alpha,n)$.

\begin{thm} \thlabel{thm:pyber}
	Let $\pi$ be a random permutation with law $ESF(\alpha,n)$. Then the probability that $\pi$ belongs to a transitive subgroup other that $A_n$ or $S_n$ is $o(1)$. 
\end{thm}

To prove this we need three claims, and a lemma. 

\begin{claim}
Let $L = \sum_{\ell =1}^{n} C^{(n)}_\ell$ be the total number of cycles in $\pi$. With probability at least $1 - o(n^{-\alpha/2})$ it holds that
$|L - \alpha \log n| \leq  (\alpha/10) \log n$. 
\end{claim}

\begin{proof}
\cite[Theorem 3.7]{arratia2000} states that $L \overset{TV} \to \Poi(\alpha \log n).$ 
%
%
The claim immediately follows from the standard tail estimate (that comes from the Chernoff inequality) for a Poisson random variable
$$\max\{\mathbf P[ \Poi(\lambda) \leq x],\mathbf P[ \Poi(\lambda) \geq x]\}\leq {\frac {e^{-\lambda }(e\lambda )^{x}}{x^{x}}} .$$
%
%
%
\end{proof}

\begin{proof}[Proof of \thref{claim:claim2}]

The theorem statement is such that if it holds for $\omega(n)$, then it also holds for any $\omega'(n) \geq \omega(n)$. So, it suffices to prove the statement for all $\omega(n) \leq \sqrt[4]{\log n}.$ 

Fix such an $\omega(n)$ and set $b_n = n \exp(- \omega(n) \sqrt{\log n}).$ Let $P$ be the set of all primes and define $P_n = P \cap [b_n,n]$. 
Given $p \in P_n$ let $I_p = [1,2,\hdots, \lfloor n/p \rfloor ]$ be the set of all $i$ for which $pi$ is a multiple of $p$ between $p$ and $n$. 
We also introduce the entire collection of such multiples smaller than $n$ 
$$W_n = \bigcup_{ p \in P_n} \{(p,i) \colon i \in I_p\}.$$

It suffices to show that the event
$$E_n  = \{\exists (p,i) \in W_n \colon ip \mid \Phi(\pi)\}$$
occurs with high probability. We can write $E_n$ as a union of events, 
$$E_n = \bigcup_{(p,i) \in W_n}  \{C_{ip}^{(n)} >0\}.$$
Notice that any two numbers in the set of products $$\{ip\colon (p,i) \in W_n \}$$ are different. To see this we argue by contradiction, suppose that $ip=i'p'$ with $p < p'$. Because of primality, this could only happen if $p'\mid i$. Thus $i \geq p'$. Also note that our assumption $\omega(n) \leq  \sqrt[4]{\log n}$ ensures $b_n > \sqrt n$. Thus, $p,p' > \sqrt n$.
These two inequalities give $$ip \geq p'p > \sqrt n \sqrt n = n.$$ 
But $ip > n$ is a contradiction, since our construction of 
$I_p$ ensures that $ip \leq n$. 

We have now shown that the occurrence of $E_n$ is equivalent to the sum of cycle counts being positive: $E_n = \{ \sum_{W_n} C_{ip}>0 \}$.
The Feller coupling \eqref{eq:feller} ensures that 
\begin{align}
\sum_{W_n} C_{ip} \geq -D_n + \sum_{W_n} Y_{ip}, \label{eq:coupling} 
\end{align}
with $D_n$ the total number of deletions. It follows that 
\begin{align}
\P[E_n] \geq \P[ -D_n + \sum_{W_n} Y_{ip} >0]. \label{eq:En}	
\end{align}


%
Using additivity of independent Poisson random variables, the sum on the right of \eqref{eq:coupling} is distributed as a Poisson with mean
\begin{align}\mu_{n,\alpha}= \sum_{p \in P_n} \sum_{i=1}^{\lfloor n/(p)\rfloor} \alpha/ip &\approx \alpha \sum_{p \in P_n}\f{\log(n/(p))} p  \nonumber \\
&= \alpha \Bigl(  \log n \sum_{ p \in P_n} \f{1}{p} - \sum_{p \in P_n} \f{\log p}{p} \Bigr).\label{eq:mu} 
\end{align}

\thref{lem:mu} shows that $\mu_{n,\alpha} \to \infty$.  
To establish that \eqref{eq:En} occurs with high probability, we will show that the number of deletions, $D_n$, in the Feller coupling is unlikely to remove every successful $Y_{ip}$ in \eqref{eq:coupling}.
  
As $\mu_{n,\alpha} \to \infty$, we take any $h(n)= o(\mu_{n,\alpha})$ with $h(n) \to \infty.$ By \thref{lem:feller} \eqref{eq:D} we know that $\E D_n \leq c$ for some $c= c(\alpha)$. It follows from Markov's inequality that 
$$\P[D_n \geq h(n)] \leq \f{c}{h(n)}= o(1).$$
Standard estimates on a Poisson random variable tell us that
$$\P[\Poi(\mu_{n,\alpha}) \leq h(n)] = 1 - o(1).$$
So $D_n$ is with high probability smaller than $\sum_{W_n} Y_{ip}$, and combined with \eqref{eq:coupling} we have established that the right side of \eqref{eq:En} occurs with high probability. 
\end{proof}

\begin{lemma} \thlabel{lem:mu}
Let $\mu_{n,\alpha}$ be as in the proof of \thref{claim:claim2}. It holds that $\mu_{n,\alpha} \to \infty$. 
\end{lemma}
\begin{proof}
It suffices to show that \eqref{eq:mu} tends to infinity. Two classical theorems attributed to Euler (\cite{euler}) and Chebyshev (\cite{apostol}), respectively, state that for sums of primes smaller than $n$
\begin{align*}
	\sum_{p \leq n} \f 1p &= \log \log n+O(1) \quad \text{ and } \quad 
	\sum_{p \leq n} \f {\log p}{p} = \log n+O(1).
\end{align*}
	Recalling that $P_n = P \cap [b_n, n]$, it follows that \eqref{eq:mu} is asymptotic to
$$  \log n \log \f{\log n}{\log b_n} -\log \f {n} {b_n}.$$
We can express the second term precisely 
\begin{align}
\log \f {n}{b_n} = \log \exp \left(\omega(n)\sqrt{\log n} \right) = \omega(n) \sqrt{\log n}.\label{eq:secterm}
\end{align}
 Similarly we compute the order of the first term. Notice
  \begin{align}
\log \f{\log n}{\log b_n}&= \log \Bigl( \f {\log n}{\log n -\omega(n) \sqrt{\log n} } \Bigr) \nonumber \\
&= \log \Bigl( 1+ \f {\omega(n) \sqrt{\log n}}{\log n -\omega(n) \sqrt{\log n}}\Bigr). \label{eq:log(1-x)}
\end{align}
As $\omega(n) \leq \sqrt[4]{\log n}$ the above term is $\log(1 - o(1))$.
Using the fact that $\log(1-x) \approx -x$ as $x \to 0$ we have \eqref{eq:log(1-x)} is asymptotic to
\begin{align*}
\f {\omega(n) \sqrt{\log n}}{\log n -\omega(n) \sqrt{\log n}}.
\end{align*}
Reintroducing the $\log n$ factor, we now have
\begin{align}
	\log n \sum_{P_n} \f 1p \approx \log n \f { \omega(n) \sqrt{\log n}}{\log n -\omega(n) \sqrt{\log n}}. \label{eq:log5/8}
\end{align}
Combining \eqref{eq:secterm} and \eqref{eq:log5/8} gives
\begin{align*}
\mu_{n,\alpha}&\approx \log n \f { \omega(n) \sqrt{\log n}}{\log n -\omega(n) \sqrt{\log n}}- \omega(n) \sqrt{\log n}\\
&= \omega(n)^2\f{\log n}{\log n -\omega(n) \sqrt{\log n}}\\
&=\omega(n)^2 (1- o(1))
\end{align*}
Since $\omega(n)\to \infty$ we conclude that $\mu_{n,\alpha}\to \infty$.

\end{proof}

Now we set our sights on proving \thref{thm:pyber}. This will require \thref{thm:bovey} and \thref{claim:claim2}, as well as a few additional lemmas. 

\begin{lemma}\thlabel{claim:claim3}  
Let $0<\gamma<1$ and $\psi(n) \rightarrow \infty$ such that $\psi(n) = o (\log n)$. The probability that simultaneously for all $2\leq r \leq \exp( \psi(n))$ there exists $C_{j_r}>0$ such that $j_r>n^{\gamma}$ and $r \nmid j_r$ is $1- o(1)$.
\end{lemma}
\begin{proof}
Let $r_0 = \exp(\psi(n))$ and consider the complement of the event in question:
$$E =\{\exists r < r_0 \colon r \mid j \text{ for all $  n^\gamma< j$ such that $C_j > 0$}\}.$$
It suffices to show that $\P[E] = o(1)$. In order to model the values $C_j$ with independent Poisson random variables we will work with a larger event. For any $\gamma < \gamma' < 1$, we consider $$E_{\gamma'} = \{\exists r < r_0 \colon r \mid j \text{ for all $  n^\gamma < j < n^{\gamma'}$ such that $C_j > 0$}\}.$$
  A union bound leads to 
\begin{align}
	\P[E_{\gamma'}] &\leq \sum_{r=2}^{r_0} \P[ r \mid j \text{ for all } n^\gamma < j < n^{\gamma'}\text{ such that } C_j > 0] \nonumber \\
		&= \sum_{r=2}^{r_0} \P[ C_j =0 \text{ for all $n^\gamma < j < n^{\gamma'}$ such that $r \nmid j$}]. \label{eq:r_bound}
\end{align}
Let $A_r$ be the set of integers in $(n^\gamma,n^{\gamma'}]$ that $r$ fails to divide. Formally,
$$A_r = [n^{\gamma},n^{\gamma'}] \cap ([n] \setminus \{ ir \colon i\leq \lfloor n/r \rfloor\}).$$
%
 As $n^{\gamma'} = o(n)$ we can apply \cite[Theorem 3.1]{arratia2000} to conclude that 
$C[n^\gamma, n^{\gamma'}] \overset{TV}\to Y[n^\gamma, n^{\gamma'}].$
Thus, the probabilities at \eqref{eq:r_bound} converge to
\begin{align}
	\sum_{r=2}^{r_0} \prod_{j \in A_r} \P[Y_j = 0] &=	\sum_{r=2}^{r_0} \exp\Bigl(-  \sum_{j \in A_r} \f{\alpha}{j} \Bigr). \label{eq:r_bound2}
\end{align}
The quantity $\sum_{j \in A_r} \f \alpha j$ is minimized at $r=2$,
 and in this case $\sum_{j \in A_2} \f 1 j \gg \log n$. We can apply this to \eqref{eq:r_bound2} to obtain
$$\P[E_{\gamma'}] \leq \sum_{r=2}^{r_0} n^{-\delta} \leq r_0 n^{-\delta} = o(1),$$
for some $\delta >0$.  As $E \subseteq E_{\gamma'}$, this completes the proof.
%
\end{proof}

\begin{lemma}
\thlabel{lem:pik}
Let $\delta = \delta(\epsilon,\alpha) = 1- (\alpha + \epsilon)\log 2.$ For any $\epsilon >0$, there exists a constant $c = c(\epsilon, \alpha)$ such that the probability of $\pi$ fixing a set of size $k$ is bounded by $c k^{-\delta}$ uniformly for all $1 \leq k \leq n/2$.  	
\end{lemma}

\begin{proof}

In \thref{lem:quenched} we show that $\P[k \in \mathcal L(\mathbf X(\alpha))]\leq  C k^{-\delta}$ for some $C$ and all $k \leq n$. 
The Feller coupling lets us relate the probability of a fixing a set size to this quantity, the only complication is a possible contribution from $\ind{J_n = \ell}$. However,  the location of this extra cycle is unconcentrated enough to not alter the order of these probabilities. 

\thref{lem:feller} \eqref{eq:J1} ensures that $\P[J_n = \ell] \leq \f{\alpha}{n-\ell}.$ We make use of this as we condition on the value of $J_n$. 
\begin{align}
\P[\pi \text{ has a $k$-cycle}] & \leq 
\P[k \in \mathcal L(\mathbf X (\alpha))] + \sum_{\ell \leq k} \P[ k - \ell \in  \mathcal L(\mathbf X (\alpha))] \P [J_n = \ell] \\
& \leq C k^{-\delta} + \sum_{\ell \leq k} C(k-\ell)^{- \delta} \f{\alpha}{n-\ell} \label{eq:bound2}\\
&\leq C k^{-\delta} + \f{2 \alpha}{n} \sum_{\ell \leq k} C\ell^{-\delta} \label{eq:bound3}\\
& \leq C k^{-\delta} + C' n^{-\delta}. \nonumber
\end{align}
Where at \eqref{eq:bound2} we use our bound on $\P[k \in \mathcal L(\mathbf X(\alpha))] $ from \thref{lem:quenched} and bound $\P[J_n = \ell]$ with \thref{lem:feller} \eqref{eq:J1}. The next line, \eqref{eq:bound3}, reindexes the sum and uses that $\ell \leq n/2$ to bound $\frac{\alpha}{n-\ell} \leq 2\alpha/n$. Our claim then follows by bounding the leftmost sum by $n^{-\delta}$ time some constant $C'$, and the fact that $k^{-\delta} > n^{-\delta}$.
\end{proof} 


We now prove a result analogous to \cite[Theorem 1]{luczak}. Namely, that if a permutation sampled according to Ewen's sampling formula is in a transitive subgroup, this subgroup is likely to be either all of $S_n$ or the alternating group, $A_n$.  Though the lemmas culminating to this used rather different arguments, our proof from here closely follows that in \cite{luczak}. Recall that a \emph{primitive subgroup} is a transitive subgroup that does not fix any partition of $[n]$.

\begin{proof}[Proof of \thref{thm:pyber}] Let $\pi\in G$ for a transitive subgroup $G$. According to \cite{guralnick1998minimal}, it follows from results of \cite{babai1981order} that the minimal degree of a primitive subgroup not containing $A_n$ is at least $(\sqrt{n}-1)/2$. 
On the other hand, by \thref{thm:bovey} with high probability the minimal degree of $\langle \pi \rangle$ is less than $n^{0.4}$. Thus, the probability that $G$ is primitive is $o(1)$. 

Assume $G$ is a non-primitive, transitive subgroup. Then it must fix some partition $\{B_i\}_{i=1}^r$ of $[n]$. 
By transitivity every block in such a partition will be mapped to every other, thus the partition must be into blocks of equal size, that is $$|B_i|=\f n r .$$
Note that any cycle of a permutation that preserves this partition must have the same number of elements in each block it acts on. We will use this to establish that any $\pi\in G$ with high probability cannot fix this partition for any $r$. 

We consider three cases for the possibly amount, $r$, of blocks:
\begin{description}
\item[Case 1] $2 \leq r \leq r_0 = \exp (\log \log n \sqrt{\log n})$ \quad 


	By \thref{claim:claim3} with $\psi(n) = \log \log n \sqrt{\log n}$, with high probability for every $r$ there exists there exists $j_r > n^{.99}$ such that $\pi$ contains a cycle of length $j_r$ and $r\nmid j_r$.
	Fix some $r$ and $j_r$, and let $B$ be the union of all blocks this cycle acts upon. 
	Because $r$ does not divide $j_r$, $B$ is a proper invariant subset of $[n]$ with size ${sn}/r$, where $s$ is the number of blocks in $B$. 
	By \thref{lem:pik} this occurs with probability bounded by
	\[
		\sum_{r=2}^{r_0}\sum_{s=1}^{r} c (sn/r)^{-\delta} \leq c r_0^2 (n/r_0)^{-\delta} = o(1).
	\]

\item[Case 2] $r_0 \leq r \leq n/r_0$ \quad

	By \thref{claim:claim2}, with high probability there exists a prime $q > n/r_0$ that divides $\Omega(\pi)$. 
	Thus there exists a cycle of length $j$ such that $q \mid j$. 
	Let $t\leq r$ be the number of blocks this cycle intersects. Because of the bounds on $r$ we have
	\[
	\f {j}{t}\leq |B_i|=\f n r\leq \f n {r_0}< q.
	\] 
Since $q$ divides $j=t \f {j}{t}$ and $\f{j}{t} < q$ by the above inequality, we must have that $q\mid t$. 
	However, the total number of blocks $r$ satisfies
	\[ 
	r< \f n r_0< q
	\]
	so this is impossible. 

\item[Case 3] $n/r_0 \leq r < n$ \quad 

Let $a$ be as in Proposition \ref{prop:gcd} and $\epsilon>0$ such that $a+\epsilon<1$. 
Using \thref{claim:claim3} with $\psi (n) = \log \log n \sqrt{\log n}$ one more time, we now find $j > n^{a+\epsilon}$ such that $\pi$ contains a cycle of length $j$ and the blocks size, $s=n/r$ does not divide $j$. 
Again, let $B$ be the union of blocks which intersect this cycle. 
Because the size of the blocks does not divide $j$, this intersection is a proper subset of $B$. 
We conclude that $\pi$ must contain another cycle of length $j'$ that intersects each block in $B$. 
Hence both cycle lengths $j$ and $j'$ are divisible by the number of blocks in $B$. 
We bound the number of blocks,
\[
\f{|B|}{s}>\f{j}{s}>\f{n^{a+\epsilon}}{s}>n^{a}.
\]
However, by Proposition \ref{prop:gcd} with probability $1- o(1)$ no two cycle lengths have a common divisor greater than $n^a$. This completes our proof. 
\end{description}


\end{proof}

\subsection{Proving \thref{ESF}} \label{sec:ESF}
To rule out generating $A_n$ we need an estimate that shows an $\ESFn$ permutation is odd with probability bounded away from zero. Of course this probability ought to converge to $1/2$, but we were unable to find a proof. We make due with an absolute lower bound.

\begin{lemma}\thlabel{lem:odd}
Let $\pi \sim ESF(\alpha,n)$, then 
$$\inf_{n \geq 2} \min \{\P[\pi \text{ is odd}] , \P[\pi \text{ is even}] \}\geq   \min\left\{ \f 1{\alpha+1}, \f{\alpha}{\alpha+1}\right\}.$$ 
\end{lemma}

\begin{proof}
We will proceed by induction. It is straightforward to check via the Feller coupling that an $\ESF(\alpha,2)$ permutation is odd with probability $\f{\alpha}{\alpha+1}$, and it is even with probability $1- \f{\alpha}{\alpha+1} = \f 1 {\alpha+1}$. 

Suppose that for $2\leq m \leq n-1$, if $\pi_m \sim ESF (\alpha, m)$ then
 		$$\min\{\P[\pi_m \text{ is odd}] , \P[\pi_m \text{ is even}] \}\geq  \epsilon,$$
 		for some yet to be determined $\epsilon>0$. 
 		Now, it suffices to prove that 
 		\begin{align}
 			\min\{\P[\pi \text{ is odd}] , \P[\pi \text{ is even}] \}\geq  \epsilon. \label{eq:inductive}
		\end{align}
Consider the sequence $\xi_1,\xi_2,\hdots$ of Bernoulli random variables for the Feller coupling. 
Recall that $J_n = (n+1) - R_n$, with $R_n$ the index of the  rightmost 1 in $\xi_1,\hdots,\xi_{n}$. 
We will condition on the value of $J_n$, but need to also take into account the parity of $n$.

\begin{description}
	\item[$n$ is odd] We  condition on the value of $J_n$. For convenience let $q_\ell = \P[J_n = \ell].$ 
	If $J_n = \ell$, then $\xi_{n+1-\ell} = 1$ and we can view the beginning sequence $\xi_1, \ldots, \xi_{n-\ell}, 1$ as corresponding to a random permutation sampled with distribution $ESF(\alpha, n-\ell)$.
	We denote such a permutation by $\pi_{n-\ell}$.
	Noting that when $J_n = n$, $\pi$ is a single $n$-cycle and therefore even, it follows that
	\begin{align}
	\P[\pi \text{ is odd}] 	&=\sum_{\ell \text{ odd}}^{n-3} \P[\pi_{n -\ell} \text{ is even}]q_\ell  + \sum_{\ell \text{ even}}^{n-2} \P[\pi_{n -\ell} \text{ is odd}] q_\ell+ q_{n-1} .\nonumber
\end{align}
%
By the inductive hypothesis 
$$	\P[\pi \text{ is odd}] \geq \sum_{\ell =1}^{n-2}\epsilon q_\ell + q_{n-1}.$$
Since $\sum_1^n q_\ell = \sum_1^n \P[J_n = \ell] =1$, it suffices to prove that $q_{n-1} \geq  \epsilon( q_{n-1} + q_n)$. Equivalently, this requires that $(1-\epsilon) q_{n-1} \geq \epsilon q_n.$
Notice that $$q_{n-1} = \P[\xi_2 = 1]\prod_{i=3}^n \P[\xi_i = 0], \text{ and } q_n = \prod_{i=2}^n \P[\xi_i = 0].$$
Much of this cancels when we solve $(1-\epsilon) q_{n-1} \geq \epsilon q_n$, so we arrive at the requirement 
$$(1- \epsilon) \f{\P[\xi_2 = 1]}{\P[\xi_2=0]} \geq \epsilon.$$
The ratio $\f{\P[\xi_2 = 1]}{\P[\xi_2=0]}=\alpha$, and thus we require
\begin{align}
	(1-\epsilon)\alpha \geq \epsilon. \label{eq:alpha1}
\end{align}

\item[$n$ even] The argument is similar. The key difference is that we instead need $ q_{n} \geq \epsilon( q_{n-1}+ q_n)$. Rewriting as before this is the same as 
\begin{align}
	\epsilon \alpha \leq (1- \epsilon)\label{eq:alpha2}.
\end{align}
\end{description}
After solving for $\epsilon$ in \eqref{eq:alpha1} and \eqref{eq:alpha2}, we see that \eqref{eq:inductive} holds whenever
$0< \epsilon \leq \min\left\{ \f 1{\alpha+1}, \f{\alpha}{\alpha+1}\right\}.$
%
\end{proof}
\begin{remark}
	When $\alpha =1$ this confirms the intuition that even and odd permutations occur with probability $1/2$ for all $n \geq 2$. 
\end{remark}


\begin{cor}\thlabel{cor:odd} Let $\delta= \delta(\alpha,\epsilon)$ and $c = c(\alpha, \epsilon)$ be as in \thref{lem:pik}. 
It holds for 
all $k \leq n/2$ that
	$$\P[\pi \text{ fixes a $k$-set} \mid \pi \text{ is odd}] \leq c  k^{-\delta} \cdot \max\left\{\f {\alpha +1}{\alpha}, \f {\alpha+1}{1} \right\}.$$
\end{cor}

\begin{proof}
We start with the reformulation
$$\P[\pi \text{ fixes a $k$-set} \mid \pi \text{ is odd}]  =\f{ \P[\pi \text{ fixes a $k$-set and $\pi$ is odd}  ]} {\P[\pi \text{ is odd}] }.$$
The numerator is bounded by $\P[\pi \text{ fixes a $k$-set}]$ which is less than $c k^{-\delta}$ by \thref{lem:pik}. 
We bound the denominator way from zero by \thref{lem:odd}, completing our proof. 
\end{proof}

\section{The upper bound for $s_\alpha$}  \label{sec:upper_bound}
We carry out the plan described in Section \ref{sec:overview}. First we bound the quenched probability $\tilde p_k$ defined formally just below here. The proof is similar to \cite[Lemma 3.1]{ppr}. We repeat a selection of the details for clarity, and point out where the argument differs with general $\alpha$. Subsequently, we mirror \cite[Theorem 3.2]{ppr} as we apply the Borel Cantelli lemma. 

Recall the functions $f_k$ and $g_k$ from \eqref{eq:def}. Namely that $f_k$ is the number of available summands smaller than $k$, and $g_k$ is the largest sum attainable with them. Also, define \begin{align}
	\tau_\epsilon &:= \sup \{ k \colon f_k \geq (\alpha+ \epsilon) \log k\}, \nonumber \\
	\tau &: = \sup \{ n \colon g_{\lambda_k} \geq k\} \nonumber,
\end{align}
where $\epsilon >0$ is yet to be chosen, but small, and $\lambda_k := \left\lfloor \frac{k}{  \alpha \log k} \right \rfloor$.  
The proofs of \cite[Lemma 2.1, Lemma 2.2]{ppr} generalize in a straightforward way, so that $T = T(\epsilon,\delta) := \max \{\tau_\epsilon, \tau\}$ is almost surely finite. As $\P[T > \lambda_k] \to 0$, we define the quenched probabilities to avoid this diminishing sequence of events
$$\tilde p_k := \P[ T < \lambda_k \text{ and } k \in \L(\X) ].$$

\begin{lemma} \thlabel{lem:quenched}
Suppose that $\alpha < 1 / \log 2$. For each $\epsilon >0$ there exists $C=C(\epsilon)$ such that for all $k\geq 1$, 
$$\tilde p_k(\alpha) \leq Ck^{- 1 + (\alpha+ \epsilon) \log 2 }.$$ 
\end{lemma}
\begin{proof}


Fix $\epsilon>0$. 
Let $G_k$ be the event that $T < \lambda_k$ while also $n \in \L(\X)$. Call a sequence $\mathbf y = (y_1,y_2,\hdots)$ admissible if it is coordinate-wise less than or equal to $\X$. When $G_k$ occurs, it is not possible for $n$ to be a sum of $\sum_j j y_j$ for an admissible $\mathbf y$ with $y_j$ vanishing for $j > \lambda_k$. Indeed, on the event $G_k$, we have $\sum_{j \leq \lambda_k}j X_j < n.$ The event satisfying the following three conditions contains $G_k$:
\begin{enumerate}[label = (\roman*)]
	\item $f_{\lambda_k} \leq (\alpha + \epsilon) \log \lambda_k$.
	\item $g_{\lambda_k} < n$.
	\item There is some $k$ with $k = \sum_j (y_j' + y_j'')$ with $\mathbf y '$ supported on $[0,\lambda_k]$ and $\mathbf y''$ nonzero and supported on $[\lambda_k+1,k]$ with both $\mathbf y'$ and $\mathbf y''$ admissible. 
\end{enumerate}
Define the probability $p'_k = \P[f_{\lambda_k} \leq (\alpha + \epsilon) \log \lambda_k\text{ and }g_{\lambda_k} < n \text{ and } \ell = \sum_j jy_j']$ where $\mathbf y'$ is admissible and supported on $[1, \lambda_n].$ Also, define 
$p''_\ell = \P[\textstyle \ell =\sum_j jy''_j]$ where $\mathbf y''$ is admissible and supported on $[\lambda_k+1,k].$ Using independence we can obtain an upper bound on $\tilde p_k$ by decomposing into these two admissible vectors:
\begin{align}
\tilde p_k &\leq \sum_{k=\lambda_k +1}^k p'_{k-\ell} p''_\ell \nonumber \\
&\leq \left( \sum_{\ell	=\lambda_k +1}^k p'_{k-\ell} \right) \max_{ \lambda_k \leq \ell \leq k} p''_\ell. \label{eq:3.3}
\end{align}
The first factor above is (by Fubini's theorem) equal to $\E| \L(\X) \cap [0,\lambda_k]|$. Using the trivial bound where we assume each sum from $\X$ is distinct, and that we are working on the event $f_{\lambda_k} \leq (\alpha + \epsilon) \log \lambda_k$ we have
\begin{align}|\L(\X) \cap [0,\lambda_k]| \leq 2^{ Z_m} \leq \lambda_k^{(\alpha + \epsilon) \log 2} \label{eq:p'}.	
\end{align}
We now bound the second term. We will show there exists a constant $C$ such that
\begin{align}
p''_\ell \leq C \f {\log^2 k 	}{k},  \text{ for all } \ell \in [\lambda_k+1, k]. \label{eq:3.5}
\end{align}
Let $H$ be the event that $X_j \geq 2$ for some $j \in [ \lambda_k,k]$. As $\P[X_j \geq 2] \leq \alpha/j^2$ we can write
\begin{align}
\P[H] \leq \sum_{j = \lambda_k}^k \alpha/j^2 \leq \alpha / \lambda_k \ll \f{ \log k }k	.
\end{align}
The event that $\ell = \sum_j j y''_j$ for an admissible $\ell$ supported on $[\lambda_k, k]$, but that $H$ does not occur is contained in the union of events $E_j$ that $X_j =1$ and $\ell=j = \sum_i i y_i''$ for some admissible $\mathbf y''$ supported on $[\lambda_k,k] \setminus \{j\}$. Using independence of $X_j$ from the other coordinates of $\X$, along with our description of what must happen if $H$ does not, we obtain
\begin{align}
	p''_\ell &\leq \P[H] + \sum_{j= \lambda_k} ^k \f 1 j p''_{\ell-j} \nonumber \\
	&\leq \P[H] + \f 1 { \lambda _k} \sum_{ j = \lambda_k} ^k p_{\ell-j}'' \nonumber \\
	&\leq \f {\log k }{ k } \left( 1 + \sum_{ j = \lambda_k}^k p''_{\ell-j} \right) \label{eq:3.6}. 
\end{align}

It remains to bound the summation in the last factor. Here taking $\alpha \neq 1$ is a minor nuisance, but we circumvent any difficulties with the crude bound $\alpha \leq 2$. The idea is to use the generating function $$F(z,\X):= \prod_{j \in [\lambda_k,k] : X_j = 1} ( 1 + z^j).$$
Writing $[z^j]F$ for the $j$th coefficient of $F$. Observe that this corresponds to forming $j$ using the $j \in [\lambda_k,k]$ for which $X_j=1$. It follows that 
$$\sum_{ j= \lambda_k}^k p''_{\ell-j} \leq \sum_{j = \lambda_k}^k \E \left ( [z^{\ell-j}]F(z,\X) \right) .$$
Using that $[z^j]F(z,\X) \leq F(1,\X)$ we can bound the above line by $\E F(1, \X)$. This is tractable using independence of the $X_j$. Carrying the calculation out we obtain
\begin{align*}
\E F(1, \X) &= \prod_{j= \lambda_k} ^n \E [1 + \ind{X_j =1}] 	\\
&= \prod_{j= \lambda_k} ^k \left( 1 + \f \alpha j \right) \\
&\leq \prod_{j= \lambda_k} ^n \left( 1 + \f 2 j \right). 
\end{align*}
The last term can be expanded and written as
$$\f{k^2 +3k + 2}{\lambda_k^2 + \lambda_k} \ll \f { k^2}{ \lambda_k^2} \ll \log ^2 k.$$
Combine this with \eqref{eq:3.6} and we have 
$$p_\ell'' \ll \f { \log k }{ k } \left(1 + \log ^2 k  \right)  \ll \f { \log^3 k }{ k }.$$
This yields \eqref{eq:3.5}. To conclude we plug \eqref{eq:p'} and \eqref{eq:3.5} into \eqref{eq:3.3} and arrive at
$$\tilde p_k \leq C( k -\lambda_k)^{(\alpha + \epsilon) \log 2 } \f { \log^3 k}{k}.$$
This is bounded above by a constant multiple of $k^{ -1 + (\alpha+ 2 \epsilon) \log 2 }$, which, after swapping all $\epsilon$'s for $\epsilon/2$, establishes the lemma. 
	
\end{proof}

\begin{prop} \thlabel{prop:upper}
Let $s_\alpha$ be as in \thref{thm:main} and $h(\alpha)$ be as in \eqref{eq:h}. For continuity points of $h$ it holds that $s_\alpha \leq  h(\alpha)$, and $s_\alpha \leq h(\alpha)+1$ for all $\alpha$. 
\end{prop}
\begin{proof}[Proof of \thref{prop:upper}]
Fix a small $\epsilon >0$ and let $s_{\alpha,\epsilon} = \left \lceil ( 1 - (\alpha-2\epsilon) \log 2)^{-1} \right \rceil.$  Sample $s_{\alpha,\epsilon}$ vectors $\X^{(1)}, \hdots, \X^{(s_{\alpha,\epsilon})}$ and let $T(\X^{(1)}),\hdots,T(\X^{(s_{\alpha,\epsilon})})$ be the a.s.\ random times distributed as $T$. 
Take $T^* = \max \{ T(\X^{(i)}) \}_1^{s_{\alpha,\epsilon}}.$ 
By independence of the $\X^{(i)}$ and the bound in \thref{lem:quenched} we have the probability that $T^* < \lambda_k$ and $k \in \L(\X^{(i)})$ for all $1 \leq i \leq s_{\alpha,\epsilon}$ is at most a constant multiple of $k^{s_{\alpha,\epsilon} (-1 + (\alpha+ \epsilon) \log 2 ) }.$ 
Our choice of $s_{\alpha,\epsilon}$ ensures that the exponent is less than $-1$. The series is thus summable. It follows that finitely many $k > T^*$ belong to the intersection of the $\L(\X^{(i)})$. 
As $T^*$ is almost surely finite (the proof of this can be found in \cite[Section 2]{ppr}), the result follows by letting $\epsilon \to 0$ so that $s_{\alpha,\epsilon} \to h(\alpha)$.   
\end{proof}

%
%

\section{The lower bound for $s_\alpha$} 
\label{sec:lower_bound}

Here we establish the lower bounds on $m_\alpha$. 

\begin{prop}\thlabel{prop:lower}
Let $s_\alpha$ be as in \thref{thm:main} and $h(\alpha)$ be as in \eqref{eq:h}. For $0 \leq \alpha < 1 /\log 2$, it holds that $s_\alpha \geq h(\alpha)$.  
When $\alpha \geq 1/ \log2$ it holds that $s_\alpha = \infty$.
\end{prop}

We prove this by a Fourier-analytic argument.
We begin by considering continuity points $\alpha$ of $h$, and set $m = h(\alpha)$.
We then fix $m$ i.i.d.\ Poisson vectors $\X^{(1)},\ldots,\X^{(m)}$.

\begin{lemma}\label{lem:Green3.2}
  Let $I=(0,k]$ and let $\ve>0$. Then the event
  \[
    A=\left\{\bigcup_{i=1}^m \L(I,\X^{(i)})\subset \left[0,\frac{mk}{\alpha\ve}\right]\right\}
  \]
  holds with probability $\P[A]\geq 1-\ve$.
\end{lemma}
\begin{proof}
  This is a straightforward application of Markov's inequality.
\end{proof}

Recall our notation that $\X^{(i)}=\bigl(X^{(i)}_j\bigr)_{i=1}^{\infty}$. 
\begin{lemma}\label{lem:Green3.3}
  Fix $\ve\in (0,1/2)$. For any $\delta>0$, there is a constant $C=C(\ve,\delta,m,\alpha)$ such that the event
  \begin{equation}\label{eqn:eventEdefine}
    E=E(\ve,\delta,m,\alpha):=\left\{\min_{i}\sum_{\ell<j\leq k}X_j^{(i)}\geq (1-\delta)\alpha\log(k/\ell)-C\quad \forall 1\leq \ell\leq k\right\}
  \end{equation}
  holds with probability $\P[E]\geq 1-\ve$.
\end{lemma}
We remark that this is a straightforward generalization of \cite[Lemma 3.3]{efg}.
\begin{proof}
  In fact, we will take $C\gg\frac{1}{\alpha\delta^2}\log(m/\ve)$ while ensuring that $C\geq 1$. Note that the claim holds vacuously when $\ell \geq e^{-C}k$, so we need only consider the case $\ell<e^{-C}k$.

  Fix $1\leq i\leq m$ and let $B$ denote the event that $\sum_{\ell<j\leq k}X_j^{(i)}<(1-\delta)\alpha \log(k/\ell)-1$ for some $\ell\leq e^{-C}k$. Writing $\ell'$ for the smallest power of $2$ with $\ell'>\ell$, we thus have
  \begin{equation}\label{eqn:partialSumIneqXj}
    \sum_{\ell'<j\leq k}X_j^{(i)}\leq \alpha(1-\delta)\log(k/\ell')
  \end{equation}
  on $B$.

  Let $\m D=\{2^i\}_{i\geq 0}$ denote the set of non-negative powers of $2$. From \eqref{eqn:partialSumIneqXj} it follows that
  \begin{equation}\label{eqn:pointwise1Bbound}
    1_B\leq \sum_{\substack{\ell'\leq 2e^{-C}k,\\\ell'\in \m D.}}(1-\delta)^{\sum_{\ell'<j\leq \ell}X_j^{(i)}-\alpha(1-\delta)\log(k/\ell')}.
  \end{equation}
  Each of the variables $X_j^{(i)}$ are independent and Poisson with mean $\E X_j^{(i)}=\frac{\alpha}{j}$. For a Poisson variable $P$ with mean $\lambda$, we have $\E a^P=e^{\lambda(a-1)}$. Moreover, the sum $\sum_{\ell'<j\leq k}X_j^{(i)}$ is Poisson with mean $\alpha\log(k/\ell')+O(1)$. Taking expectations on both sides of \eqref{eqn:pointwise1Bbound} therefore yields
  \begin{equation*}
    \P[B]\ll\sum_{\substack{\ell'\leq 2e^{-C}k,\\\ell'\in \m D.}} \exp \left[-\alpha\delta-\alpha(1-\delta)\log(1-\delta)\right]\log(k/\ell')
  \end{equation*}
  Near zero, $-\delta-(1-\delta)\log(1-\delta)=-(1+o(1))\delta^2$. Therefore
  \begin{equation*}
    \P[B]\ll \sum_{\substack{\ell'\leq 2e^{-C}k,\\\ell'\in \m D.}}(k/\ell')^{-\alpha\delta^2}\ll e^{-\alpha\delta^2 C},
  \end{equation*}
  since the sum of the geometric series is bounded by a constant multiple of its first term. Choosing $C\gg\frac{1}{\alpha\delta^2}\log(m/\ve)$ ensures that $\P[B]\leq \frac{\ve}{m}$, and now the claim follows by the union of events bound.
\end{proof}

An insight in \cite{efg} (to which they credit the idea to Maier and Tenenbaum from \cite{MT}) is that Fourier analysis may be used to obtain a lower bound on the size of the set
\begin{equation}\label{eqn:S(I,X)dfn}
  S(I;\X^{(1)},\ldots,\X^{(m)}):=\left\{(n_i-n_m)_{1\leq i<m}\colon n_i\in \L(I,\X^{(i)})\right\}.
\end{equation}
We show that the technique extends 
to the present level of generality. Note that the analysis of \cite{efg} may be regarded as the special case $m=3$ and $\alpha=1$ of our arguments.

Let $\m T=\m R/\m Z$ be the unit torus. For $\theta\in \m T$, let $\|\theta\|$ denote the distance to $\m Z$ and set $e(z)=e^{2\pi iz}$. Write $\bm \theta=(\theta_1,\ldots,\theta_m)\in \m T^m$. Define \[\m T_0^m=\{\bm \theta\in \m T^m\colon \theta_1+\cdots +\theta_m=0\}.\]
Fix an interval $I$. We define the function $F(\bm\theta)=F_I\bigl(\bm\theta,\X^{(1)},\ldots,\X^{(m)}\bigr)$ via the formula
\begin{equation}\label{eqn:Ftheta}
  F(\bm \theta):=\prod_{j\in I}\prod_{i=1}^m \left(\frac{1+e(j\theta_i)}{2}\right)^{X_j^{(i)}}.
\end{equation}
Writing $\theta_m=-\theta_1-\cdots-\theta_{m-1}$, we regard $F(\bm\theta)$ as a function on $\m T^{m-1}$. It is straightforward to verify that its Fourier transform is a function $\hat F\colon \m Z^{m-1}\to \m C$ supported on the set $S(I;\X^{(1)},\ldots,\X^{(m)})$ defined above in \eqref{eqn:S(I,X)dfn}.

The Cauchy-Schwarz inequality yields that
\begin{equation}\label{eqn:CS}
  \sum_{a\in \m Z^{m-1}}|\hat F(a)|^2\sum_{a\colon \hat F(a)\not=0}1\geq \left[\sum_{a\in \m Z^{m-1}}\hat F(a)\right]^2.
\end{equation}
Recognizing that $S(I;\X^{(1)},\ldots,\X^{(m)})\geq \bigl|\{a\colon \hat F(a)\not=0\}\bigr|$, that $\sum_{a\in \m Z^{m-1}}\hat F(a)=F(0)=1$, and that by Parseval's identity $\sum_{a\in \m Z^{m-1}}|\hat F(a)|^2=\int_{\m T^m_0} |F(\bm \theta)|^2\ d\theta$, it follows from \eqref{eqn:CS} that
\begin{equation}\label{eq:intbound}
  |S(I;\X^{(1)},\ldots,\X^{(m)})|\geq \left(\int_{\m T^m_0} |F(\bm \theta)|^2\ d\theta\right)^{-1}.
\end{equation}

Our next task is to bound the integral appearing in \eqref{eq:intbound}. We will achieve this in two steps: first, by obtaining a pointwise bound with respect to $\bm \theta$ in Lemma \ref{lem:Green3.5}, and second, by integrating out the $\bm \theta$ dependence in Corollary \ref{cor:Green3.6}.

Let $\beta\in (0,1)$ be a real parameter satisfying
\begin{equation*}
  \beta \alpha\log 2= 1-m^{-1}+\delta_2,
\end{equation*}
where $\delta_2>0$ and $\delta_2=o_\delta(1)$ is the parameter occurring in \eqref{eqn:eventEdefine}. (Note that this is possible since $\alpha$ is a continuity point of $h$, so $m<(1-\alpha\log 2)^{-1}$.) Recall the event $E=E(\ve,\delta,m,\alpha)$ from Lemma \ref{lem:Green3.3}.
\begin{lemma}\label{lem:Green3.5}
  Fix $\ve\in (0,1/2)$ and let $I$ denote the interval $I := (k^{1-\beta}, k ]$. For a tuple $\bm \theta=(\theta_1,\ldots,\theta_m)$, let $j(\bm\theta)$ denote an index for which $\|\theta_j\|$ is maximal. Then
  \begin{equation}\label{eqn:boundGreen35}
    \E_{\X^{(1)},\ldots,\X^{(m)}}1_{E}|F(\bm \theta)|^2\ll_{\ve,m} \prod_{\substack{i\colon 1\leq i\leq m,\\i\not=j(\bm\theta)}} (\|k\, \theta_i\|\vee 1)^{-1+\delta_2}.
  \end{equation}
\end{lemma}

\begin{proof}
  For each $1\leq i\leq m$, set $t_i:=k\, \theta_i$ and define the cutoff parameter 
  \[
    k_i:=\begin{cases}
      k^{\beta},& \|t_i\|\geq k^{\beta}\\
      \|t_i\|,& 1<\|t_i\|<k^{\beta}\\
      1,& \|t_i\|\leq 1.
    \end{cases}
  \]
  Note that $k_i\geq 1$ and $k_i\geq \|t_i\|^{\beta}$. For each $1\leq i\leq m$, let $Y_i$ denote the expression
  \begin{equation*}
    Y_i:=\sum_{k_i<j\leq k} X_j^{(i)}.
  \end{equation*}
  By definition of the event $E$, when it occurs we have that
  \[
    \sum_{i=1}^m Y_i\geq \alpha(1-\delta)\sum_{i=1}^m \log k_i-C(\ve).
  \]
  It will be convenient to rewrite this bound in the form
  \begin{equation}\label{eqn:1Ebound}
    1_E\ll_{\ve}\prod_{i=1}^m k_i^{-\alpha(1-\delta)\log 2}\prod_{i=1}^m 2^{Y_i}.
  \end{equation}

  We define the quantity
  \begin{equation}\label{eqn:Rdef}
    R=R\bigl(\theta,\X^{(1)},\ldots,\X^{(m)},\ve,\delta,\alpha\bigr):=|F(\bm \theta)|^2\prod_{i=1}^m 2^{Y_i}.
  \end{equation}
  To establish the bound \eqref{eqn:boundGreen35}, we will first show that $\E_{\X^{(1)},\ldots,\X^{(m)}}R\ll_{\ve,m}1$. Once this has been established, it will follow directly from \eqref{eqn:1Ebound} that
  \begin{equation}\label{eqn:boundEFtheta}
    \E1_E|F(\bm \theta)|^2 \ll_{\ve,m} \prod_{i=1}^m k_i^{-\alpha(1-\delta)\log 2}.
  \end{equation}
  Obtaining \eqref{eqn:boundGreen35} from \eqref{eqn:boundEFtheta} is a simple exercise in lower bounding the maximum of a sequence by its geometric mean. Recalling the definition of $j(\bm \theta)$ above, we have that
  \begin{equation*}
    k_{j(\bm\theta)}\geq \prod_{\substack{i\colon 1\leq i\leq m,\\i\not=j(\bm \theta)}}k_i^{\frac{1}{m-1}}.
  \end{equation*}
  Consequently
  \begin{equation}\label{eqn:prodBoundGeom}
    \prod_{i=1}^m k_i^{-\alpha(1-\delta)\log 2}\leq \prod_{\substack{i\colon 1\leq i\leq m,\\i\not=j(\bm \theta)}} k_i^{-\frac{m}{m-1}\alpha(1-\delta)\log 2}.
  \end{equation}
  Substituting the lower bound $k_i\geq \bigl(\|k\,\theta_i\|\vee 1\bigr)^{\beta}$ into \eqref{eqn:prodBoundGeom}, we see that \eqref{eqn:boundEFtheta} implies
  \begin{equation*}
    \E\, 1_E|F(\bm\theta)|^2\ll_{\ve,m}\prod_{\substack{i\colon 1\leq i\leq m,\\i\not=j}}\bigl(\|k\, \theta_i\|\vee 1\bigr)^{-\frac{m}{m-1}\alpha\beta(1-\delta)\log 2},
  \end{equation*}
  and since $\frac{m}{m-1}\alpha\beta(1-\delta)\log 2=1+o_{\delta}(1)$ by definition of $\beta$, the desired bound follows.

  Finally we justify that $\E_{\X^{(1)},\ldots,\X^{(m)}}R\ll_{\ve,m}1$. Substituting \eqref{eqn:Ftheta} into \eqref{eqn:Rdef} yields that
  \begin{equation*}
    R=\prod_{i=1}^{m} \left[2^{Y_i}\prod_{j\in I}\left|\frac{1+e(j\theta_i)}{2}\right|^{2X_j^{(i)}}\right].
  \end{equation*}
  Recall that $Y_i=\sum_{k_i<j\leq k} X_j^{(i)}$ and that $I=(k^{1-\beta},k]$. Thus
  \begin{equation*}
    R=\prod_{i=1}^m \left[\prod_{k^{1-\beta}<j\leq k_i}\left|\frac{1+e(j\theta_i)}{2}\right|^{2X_j^{(i)}}
    \prod_{k_i<j\leq k}\left|\frac{1+e(j\theta_i)}{\sqrt{2}}\right|^{2X_j^{(i)}}\right].
  \end{equation*}
  To compute $\E R$, we recall that if $P$ is a Poisson random variable with mean $\lambda$, then ${\log\E a^P=e^{\lambda(a-1)}}$.
   Since the variables $X_j^{(i)}$ are independent Poisson variables with mean $\E X_j^{(i)}=\frac{\alpha}{j}$, it follows after some simplification that
  \begin{equation}\label{eqn:Green35Key}
    \E R=\exp\sum_{i=1}^m \left[\sum_{k^{1-\beta}<j\leq k_i}\frac{\alpha\bigl(\cos (2\pi j\theta_i)-1\bigr)}{2j}+
    \sum_{k_i<j\leq k}\frac{\alpha \cos(2\pi j\theta_i)}{j}\right].
  \end{equation}

  Next we apply the following bound from \cite[Lemma 3.4]{efg}:
  \begin{equation}\label{eqn:efg34}
    \sum_{j\leq k}\frac{\cos(2\pi j\theta)}{j}=\log \min(k,\|\theta\|^{-1})+O(1).
  \end{equation}
  Substituting into \eqref{eqn:Green35Key} yields that
  \begin{equation*}
    \E R= \exp\alpha\sum_{i=1}^m\left[\frac{1}{2}\log\frac{k^{\beta}\vee\|k\,\theta_i\|}{k_i\vee\|k\,\theta_i\|}-\frac{1}{2}\log\frac{k^{\beta}}{k_i}+\log\frac{k_i\vee \|k\,\theta_i\|}{1\vee \|k\,\theta_i\|}+O(1)\right].
  \end{equation*}
  By definition of $k_i$, each summand is $O(1)$. Therefore $\E R\ll_{\ve,m}1$, as desired.
\end{proof}


Finally we integrate the bound \eqref{eqn:boundGreen35} with respect to $\bm \theta$ to obtain a bound on the integral appearing in \eqref{eq:intbound}.
\begin{cor}\label{cor:Green3.6} Let $I = (k^{1-\beta},k]$ and recall the function $F(\bm\theta)=F_I(\bm\theta)$ from \eqref{eqn:Ftheta}. Then
  \begin{equation*}
    \int_{\m T_0^m}\E 1_E |F(\bm \theta)|^2\ d\bm \theta\ll_{\ve,m} k^{1-m}.
  \end{equation*}
\end{cor}
\begin{proof}
  After the change of variables $\bm t=k\bm \theta$, the bound reduces to verifying that
  \[
    \int_{k\cdot \m T_0^m}\E 1_E |F(\bm \theta)|^2\ d\bm t\ll_{\ve,m} 1.
  \]
  Applying Lemma \ref{lem:Green3.5}, we have
  \begin{equation}\label{eqn:intSymUpper}
    \int_{k\cdot \m T_0^m}\E 1_E |F(\bm \theta)|^2\ d\bm t
      \ll_{\ve,m}
    \int_{k\cdot \m T_0^m}\prod_{\substack{i\colon 1\leq i\leq m,\\i\not=j(\bm t).}} \bigl(\|t_i\|\vee 1\bigr)^{-(1+\delta_2)}\ d\bm t.
  \end{equation}
  By symmetry, we may upper bound the integral \eqref{eqn:intSymUpper} by
  \begin{equation}\label{eqn:mUpperBound}
    m \int_{k\cdot \m T_0^m}\prod_{i=1}^{m-1} \bigl(\|t_i\|\vee 1\bigr)^{-(1+\delta_2)}\ d\bm t=m \left(\int_{-k/2}^{k/2}\left( |t_i|\vee 1\right)^{-(1+\delta_2)}\ dt_i\right)^{m-1}.
  \end{equation}
  The rightmost integral in \eqref{eqn:mUpperBound} is bounded above by
  \begin{equation*}
    m \left(2\int_{1}^{\infty}t_i^{-(1+\delta_2)}\ dt_i+2\right)^{m-1}.
  \end{equation*} 
  The desired result now follows since the latter integral is finite and independent of $k$.
\end{proof}
Substituting the result of Corollary \ref{cor:Green3.6} into \eqref{eq:intbound} lets us show $\lvert S(I;\X^{(1)},\ldots,\X^{(m)}) \rvert \gg k^{m-1}$ with high probability.


\begin{prop} \thlabel{prop:Green3.7}
	Let $I=(k^{1-\beta}, k ]$. With probability greater than $1/2$, $S(I; \X^{(1)}, \ldots, \X^{(m)}) \subset [-ck, ck]^{m-1}$ for some constant $c$ and $\lvert S(I;\X^{(1)},\ldots,\X^{(m)}) \rvert \gg k^{m-1}$.
\end{prop}
\begin{proof}
	By Lemma \ref{lem:Green3.2} with $\ve = 1 / 3$, with probability greater than $2/3$ we have $\L (I , \X^{(i)}) \subset \left[0 , \frac{3mk}{\alpha} \right] \subset [ 0 , ck]$ for some constant $c$ and all $i$. It follows that $S(I; \X^{(1)}, \ldots, \X^{(m)}) \subset [-ck, ck]^{m-1}$.
	
	By Lemma \ref{lem:Green3.3} with $\ve = 1/20$, the event $E(\ve, \delta,m,\alpha)$ holds with probably greater $19/20$ and by Corollary \ref{cor:Green3.6},
	\[ 
		\int_{\m T_0^m}\E 1_E |F(\bm \theta)|^2\ d\bm \theta\ll_{\ve,m} k^{1-m}.
	\] 
	It follows by Markov's inequality that for some sufficiently large $L$,
	\begin{align*}
		\P \left[ \left\{ \int_{\m T_0^m} |F(\bm \theta)|^2\ d\bm \theta \ll k^{1-m}\right\}^c \right] &\leq \P \left[ E^c \right] + \P \left[ 1_E \int_{\m T_0^m} |F(\bm \theta)|^2\ d\bm \theta > L k^{1-m} \right] \\
		&\leq \frac{1}{20} + \frac{\E 1_E \int_{\m T_0^m} |F(\bm \theta)|^2\ d\bm \theta}{ L k^{1-m}} \\
		&\leq \frac{1}{10}.
	\end{align*}
	
	It then follows from \eqref{eq:intbound} that with probability at least $9/10$,
	\[ 
		\vert S(I; \X^{(1)}, \ldots \X^{(m)}) \vert  \geq \left(\int_{\m T^m_0} |F(\bm \theta)|^2\ d\theta\right)^{-1} \gg k^{m-1}.
	\]
	
	Both events then occur simultaneously with probability at least $1/2$, completing our proof.
\end{proof}
	
\begin{prop} \label{prop:Green3.8}
	Let $D$ be sufficiently large constant and $I=(k^{1-\beta}, Dk ]$. With probability bounded away from zero there exists an array $\left( x_j^{(i)} \colon 1\leq i\leq m, j\in I\right)$ for which
	\[
		0 \leq x_j^{(i)} \leq X_j^{(i)}
	\]
	for all $i$ and $j \in I$, and
	\[
		\sum_{j \in I} j x_j^{(1)} = \cdots = \sum_{j \in I} j x_j^{(m)}.
	\]
\end{prop}
\begin{proof}
	We define the interval $I' = (k^{1-\beta} , k]$. By Proposition \ref{prop:Green3.7} with probability greater than $1/2$
	\[
		S(I'; \X^{(1)} , \ldots, \X^{(m)}) \subset [-ck, ck]^{m-1}
	\]
	and $\lvert S(I';\X^{(1)},\ldots,\X^{(m)}) \rvert \gg k^{m-1}$.
	
	Let $d = 2 (5/2)^{1/\alpha}$. Straightforward calculations tell us that independently with probability greater than $1/2$ there exists $j_m \in (2ck , cdk]$ such that $X_{j_m}^{(m)} > 0$. 
	Given this $j_m$, 
	there are $\gg k^{m-1}$ collections of constants $\{j_\ell\}_1^{m-1} \in (ck, (d+1)ck]^{m-1}$ such that
	\[
		(j_m - j_1, \ldots, j_m - j_{m-1}) \in S(I'; \X^{(1)}, \ldots, \X^{(m)} ).
	\]
	Therefore, independently with probability bounded away from zero, there is one such collection $\{j_\ell\}_1^{m-1}$ such that $X_{j_\ell}^{(\ell)} > 0$ for all $\ell$. 
	
	By the definition of $S(I';\X^{(1)}, \ldots, \X^{(m)})$ we have a set of values $\left(n_i \in \L (I', \X^{(i)}) \right)_{i=1}^m$ such that
	\[
		\left(n_1 - n_m \right)_1^{m-1} = (j_m - j_\ell)_1^{m-1}.
	\]
	Manipulating the terms coordinate-wise yields
	\[
		j_1 + n_1 = \cdots = j_m + n_m.
	\]
	Our claim follows by setting $D = (d+1)c$.
\end{proof}

\begin{cor} \thlabel{cor:Green3.9}
	$\bigcap_{i = 1}^m \L (\X^{(i)})$ is almost surely infinite.
\end{cor}
\begin{proof}
	We define $k_1$ to be sufficiently large for Proposition \ref{prop:Green3.8} to hold, then inductively define $k_{i+1} = \left( D k_i \right)^{1/ (1-\beta)}$. 
	The intervals $I_i = (k_i^{1-\beta}, Dk_i]$ are pairwise disjoint and the set
	\[
		\bigcap_{j=1}^m \L (I_i, \X^{(j)})
	\]
  is non-empty with probability bounded away from zero. Since these events are independent for distinct values of $i$, the claim now follows from the second Borel-Cantelli lemma.
\end{proof}
\begin{proof}[Proof of \thref{prop:lower}]
As we had fixed $m = h(\alpha)$, by \thref{cor:Green3.9} $s_\alpha \geq h(\alpha)$ holds for continuity points of $h$. 
For points of discontinuity, we fix $\alpha' < \alpha$ such that $h(\alpha') = h(\alpha) = m$.
We couple two sets of independent realizations $\X^{(1)}(\alpha), \ldots, \X^{(m)}(\alpha)$ and $\X^{(1)}(\alpha') , \ldots, \X^{(m)}(\alpha') $ such that $X^{(i)}(\alpha') \subseteq X^{(i)}(\alpha)$ for all $i$.
Now $\bigcap_{i = 1}^m \L (\X^{(i)}(\alpha'))$ is almost surely finite by \thref{cor:Green3.9}.
Our coupling ensures $\L (\X^{(i)}(\alpha' ) \subseteq \L (\X^{(i)}(\alpha))$ for all $i$, so $\bigcap_{i = 1}^m \L (\X^{(i)(\alpha')}) \subseteq \bigcap_{i = 1}^m \L (\X^{(i)}(\alpha))$. 
Thus the second intersection is almost surely infinite, and our claim follows. 
\end{proof}
\begin{proof}[Proof of \thref{thm:main}]
This now follows immediately from \thref{prop:upper} and \thref{prop:lower}.	
\end{proof}

\bibliographystyle{amsalpha}
\bibliography{sumset_percolation}
	
\end{document}